\documentclass[11pt]{article}
\newcommand{\version}{April 17, 2013 }

\usepackage{amsthm,amsfonts,amsmath, amscd}
\swapnumbers
\theoremstyle{plain}

\newtheorem{thm}{THEOREM}[section]
\newtheorem{lm}[thm]{LEMMA}
\newtheorem{cl}[thm]{COROLLARY}
\newtheorem{prop}[thm]{PROPOSITION}
\theoremstyle{definition}
\newtheorem{defi}[thm]{DEFINITION}
\theoremstyle{definition}
\newtheorem{remark}[thm]{Remark}
\newcommand{\upchi}{\raise1pt\hbox{$\chi$}}
\newcommand{\R}{{\mathord{\mathbb R}}}

\newcommand{\N}{{\mathord{\mathbb N}}}

\newcommand{\h}{{\mathcal{H}}}
\newcommand{\K}{{\mathcal{K}}}
\newcommand{\Q}{{\mathcal{Q}}}

\renewcommand{\|}{{\Vert}}
\numberwithin{equation}{section}
\def\E{{\cal E}_{N,E}}

\def\dd{{\rm d}}
\def\ncht{\left(\begin{matrix} N\cr 2\cr \end{matrix}\right)}
\def\nmcht{\left(\begin{matrix} N-1\cr 2\cr \end{matrix}\right)}

\def\STE{{\mathcal S}_{N,E}}

\def\LN{L_{N,E}}

\usepackage{color}

\begin{document}

\markboth{\scriptsize{CCL \version}}{\scriptsize{CCL \version}}

\title{SPECTRAL GAP FOR THE KAC MODEL WITH  HARD SPHERE COLLISIONS}

\author{\vspace{5pt} Eric A. Carlen$^{1}$,  Maria C. Carvalho$^{2}$, and Michael Loss$^{3}$ \\
\vspace{5pt}\small{$1.$ Department of Mathematics, Hill Center, Rutgers University}\\[-6pt]
\small{
110 Frelinghuysen Road
Piscataway NJ 08854 USA}\\
\vspace{5pt}\small{$2.$ Department of Mathematics and CMAF, University of Lisbon,}
\\[-6pt] \small{Av. Prof. Gamma Pinto 2, 1649-003 Lisbon, Portugal}\\
\vspace{5pt}\small{$3.$ School of Mathematics, Georgia Institute of
Technology,} \\[-6pt]
\small{Atlanta, GA 30332 USA}\\
 }
\date{\version}

\footnotetext [1]{Work of Eric Carlen is partially supported by U.S. NSF grant DMS 0901632. Work of  Maria Carvalho is partially supported by 
FCT Project PTDC/MAT/100983/2008.
Work of Michael Loss is partially supported by U.S.
NSF 
grant DMS 0901304\\
\copyright\, 2013 by the authors. This paper may be reproduced, in
its entirety, for non-commercial purposes.}

\maketitle

\def\dd{{\rm d}}
\def\ncht{\left(\begin{matrix} N\cr 2\cr \end{matrix}\right)}
\def\nmcht{\left(\begin{matrix} N-1\cr 2\cr \end{matrix}\right)}

\begin{abstract}

We prove the analog of the Kac conjecture for  hard sphere collisions.

\end{abstract}

\medskip
\leftline{\footnotesize{\qquad Mathematics subject classification numbers:  47A63, 15A90}}
\leftline{\footnotesize{\qquad Key Words: spectral gap, kinetic theory}}

\section{Introduction} \label{intro}

The relation between the microscopic dynamics of a system of $N$ particles, interacting through binary collisions, and the Boltzmann equation
--- which is supposed to provide a simplified description of this dynamics --- remains the source of many challenging problems in mathematics and physics. 
Since the simplified description of the microscopic dynamics  provided by the Boltzmann equation still contains all of the information
necessary for describing the hydrodynamics of the system in the appropriate scaling regimes, the Boltzmann equation
provides an essential mesoscopic link between the microscopic dynamics and the macroscopic description of the system.

In 1956 Marc Kac introduced a model \cite{K56,K59} that contains {\em only} those features of the microscopic collision process in a dilute gas
that are relevant to the derivation of the Bolztmann equation in the large $N$ limit.  The original model investigated by 
Kac involved a caricature of collisions between {\it Maxwellian molecules}, which means that the force law governing pair collisions is such that 
the rates  at which the various kinematically possible collisions take place
depends only on the angle between the in-coming and out-going relative velocities, and  not on the magnitude of the relative velocity. This affords 
considerable simplification, and most previous work on the Kac model has has been carried out in this context of Maxwellian molecules.

We  consider  a version of the Kac model where the interaction between the pairs of particles is
described by   a caricature of  hard--sphere collisions, and we solve the {\em Kac conjecture}, which is explained below,  for this more physical variant of his model.

The Kac model describes the evolution of a spatially homogenous ``gas'' of $N$ particles in terms of a stochastic process
that is a random walk, the {\em Kac walk} on the energy sphere for the $N$ particles. Each binary collision is represented by a jump to 
a new point on the energy sphere  in which only the coordinates of a single pair of particles changes.  The collision times  arrive in a Possion stream, where for each pair 
$i,j$ of particles, the jump rate is a function of $E_{i,j}$,  the kinetic energy of the pair of particles. The standard versions of the spatially homogeneous Boltzmann equation
are obtained from the Kac walk in the large $N$ limit by choosing the jump rates to be proportional to $E_{i,j}^\gamma$  for some power $0\leq \gamma \leq 1/2$.
The case $\gamma = 0$ is the case of Maxwellian molecules, in which the jump rates are uniform for all pairs, while the case $\gamma =1/2$
corresponds to hard sphere collisions, in which case the jump rates are not only non-uniform, but are not uniformly bounded from below. 

The Kac walk is a reversible process, and the uniform probability measure on the energy sphere is its invariant  measure; i.e., its
equilibrium measure. The rate at which
an initially non-uniform distribution relaxes back  to equilibrium under the Kac walk has consequences for the Boltzmann equation if and only if
this rate can be controlled uniformly in $N$. 

One way to measure this  rate is in terms of a {\em spectral gap} of the generator of the Kac walk. The Kac conjecture is that this generator has a spectral gap
that is  bounded below uniformly in $N$.

As Kac remarked, it is non-trivial even to show that the spectral gap is strictly positive for each fixed $N$. Up to 2000, the only available lower 
bounds on the spectral gap were in terms of an inverse power of $N$ \cite{DS}. 
Eventually,   Kac's conjecture was resolved 
for Maxwellian molecules in \cite{J01}, which provided no estimate on the gap,  and somewhat later in \cite{CCL00,CCL03} where the gap was computed
explicitly. 

In the case of hard spheres, the difficulty of proving the Kac conjecture is amplified by the fact that the collision rate for the pair $i,j$ 
 is  proportional to $E_{i,j}^{1/2}$. Pairs of slowly moving particles,  with a low combined energy, are effectively removed from the collision process. 
Memory of the initial distribution is only erased in the random collisions, so we need to know 
that there are not many particles that wait a long time before colliding.  The slow, low energy particles are a problem in this regard, and the 
particles with higher energy   must be shown, so to speak, to ``make up for this''.

Some results have been obtained for the Kac model with non-Maxwellian collisions. In particular, Villani \cite{V}  has studied {\em entropy production}
bounds for the Kac model with non-Maxwellian collisions, though he considers collision rates proportional to $(1+E_{i,j})^\gamma$, and in this way eliminates the
small-energy problem. However, he  proves the surprising result that  entropy production
is bounded below uniformly in $N$ for $\gamma =1$, the case of ``super hard spheres''. 
While his entropy production bound implies a spectral gap bound for rates proportional to 
$(1+E_{i,j})$, it does not seem possible to glean from this any information on uniformity of the spectral gap with rates proportional to $E_{i,j}^{1/2}$, or even $E_{i,j}$.

 In the rest of this section,  the Kac walk is defined, and  we present our main results and outline the strategy of proof.

\subsection{The Kac Walk}

For $N\in \N$ and $E > 0$,  let $\STE$ be the set
consisting of all vectors   $ v = (v_1,\dots , v_N)\in \R^N$  with
$$\frac1N\sum_{j=1}^Nv_j^2 = E \\ .$$
A point $ v\in \STE$  specifies the velocities of a collection of $N$ particles with unit mass. 
$\STE$ is the {\em energy sphere} for $N$ particles with mean energy $E$ per particle.  The Kac walk is a random walk
 in which each step of the walk corresponds to a binary collision of a single pair of particles.
 With each collision, 
 the state of the process ``jumps'' from $ (v_1,\dots , v_N)$ to
$$ (v_1,v_2,\dots,v_i^*,\dots,v_j^*,\dots,v_N)\ ,$$
where only $v_i$ and $v_j$ have changed. Since the process  models  energy conserving collisions, 
so that $\STE$ may indeed be taken as the state space of the walk, we require that
\begin{equation}\label{kinpos}
{v_i^*}^2+{v_j^*}^2 = v_i^2+v_j^2\ .
\end{equation}
The set of all collisions satisfying (\ref{kinpos}) can be parameterized by
\begin{equation}\label{colrule}
v_i^*=v_i\cos \theta+v_j\sin \theta
\qquad 
v_j^*=-v_i\sin \theta+v_j\cos \theta\ .
\end{equation}

\medskip

We now specify the rate at which these collisions occur. 
We associate to each pair $(i,j)$,  $i<j$ an exponential random variable  $T_{i,j}$ with parameter
\begin{equation}\label{jumprate}
\lambda_{i,j} =  N\ncht^{-1}(v_i^2 + v_j^2)^{\gamma}\ ,
\end{equation}
where $0 \leq \gamma \leq 1$. In more detail, the $T_{i,j}$ are independent, and
$${\rm Pr}\{ T_{i,j} > t\} = e^{-t\lambda_{i,j}}\ .$$
$T_{i,j}$ is the waiting time for  particles $i$ and $j$ to collide.
 The first collision occurs at time
\begin{equation}\label{alarm}
T = \min_{i<j}\{T_{i,j}\}\ .
\end{equation}

At the time $T$, the pair $(i,j)$ furnishing the minimum collide.  Then an angle $\theta$ is selected uniformly at random, and the process jumps from
$ (v_1,\dots , v_N)$ to
$ (v_1,v_2,\dots,v_i^*,\dots,v_j^*,\dots,v_N)$ with $v^*_i$ and $v^*_j$ given by \eqref{colrule}. After each collision, the process starts afresh.
Let $V(t)$ denote the random state of the process at time $t$.

It is possible to generalize the model to include a non-uniform rule for selecting the scattering angle $\theta$. Such generalizations are of
physical interest, and were investigated in \cite{CCL03}. However, to keep the notation simple, we restrict our attention for most of the paper
to the case in which the collision angle is chosen uniformly. This allows us to focus on the much more significant technical 
difficulties that arise from the consideration of non-uniform jump rates. 
\if false 
Indeed, most previous research  on the Kac walk  has focused on the case $\gamma= 0$, which models  the case of so-called ``Maxwellian molecules'',
as explained below, 
and in which the exponential clocks for each pair of particles all run at the same uniform rate. \fi
The case $\gamma = 1/2$ models hard sphere collision and is the case of main physical interest.

The object of our investigation is the spectral gap for the generator of the Markov semigroup associated to the Kac walk. For any continuos  function $f $ on $\STE$,
define the Kac walk generator $\LN$ by
$$\LN f ( v) = \frac1h \lim_{h\to 0}{\rm E}\{ f ( V(h)) - f( v)\ |\  V(0) =  v\ \} \ .$$
One readily computes that
\begin{equation}
\LN f ( v) =  -{N}{\ncht}^{-1}\sum_{i<j} (v_i^2+v_j^2)^{\gamma}\
 \frac{1}{2\pi}\int_{-\pi}^\pi [f ( v) - f (R_{i,j,\theta} v)] \dd \theta
\end{equation}
where
\begin{equation*}
(R_{i,j,\theta} v)_k = \begin{cases} v_i^*(\theta) & k = i\\   v_j^*(\theta) & k = j\\   v_k & k\neq i,j\end{cases}\ .
\end{equation*}
Introducing the notation
\begin{equation*}
 [f ]^{(i,j)}( v)  :=    \frac{1}{2\pi}\int_{-\pi}^\pi f (R_{i,j,\theta} v) \dd \theta\ ,
\end{equation*}
we can write the generator more briefly as
\begin{equation*}
\LN f ( v) =  -{N}{\ncht}^{-1}\sum_{i<j} (v_i^2+v_j^2)^{\gamma}\left[f ( v) - [f ]^{(i,j)}( v) \right]  \ .
\end{equation*}


Let $\mathcal{H}_E$ denote the Hilbert space  of functions on  $\STE$  that are square integrable with respect to the uniform
probability measure on $\STE$.  For $f,g\in \mathcal{H}_E$, we denote the inner product of $f$ and $g$ by $\langle f,g\rangle$ and the
norm of $f$ by $\|f\|_2$. 

Let $\E$ be the Dirichlet form on  $\mathcal{H}$, the Hilbert space  given by
\begin{equation}\label{dir}
\E(f,f) = {N}{\ncht}^{-1}\sum_{i<j}\int_{\STE}(v_i^2+v_j^2)^\gamma (f - [f ]^{(i,j)})^2\dd \sigma\ .
\end{equation}
Then
\begin{equation}\label{dir3}
\E(f,f) = -\langle f,\LN f\rangle\ .
\end{equation}
Evidently, the uniform  distribution ($f=1$) is the unique equilibrium state for this process.

\if false
is the generator of the Kac semigroup. The different values of $\gamma$ correspond to different types of collisions between the molecules in the model. The case $\gamma = 0$ corresponds to {\em Maxwellian molecular collisions}, while the case $0 < \gamma \leq 1/2$ correspond to hard 
{\em hard molecular collisions}, and in particular $\gamma = 1/2$ corresponds to {\em hard sphere
molecular collisions}, and is especially interesting on physical grounds. The case $\gamma =1$
does not correspond to any physically meaningful molecular collision model, but it does play a significant role in the mathematical investigation of kinetic theory, and is commonly to as modeling {\em super-hard interactions}. 

\fi

The object of our investigation, the {\em spectral gap}, is the quantity
\begin{equation}\label{gapdef}
\Delta_{N,E} := \inf\{ \E(f,f) \ :\ \|f\|_2 = 1\ ,\  \langle f,1\rangle =0\ ,\ f\ {\rm symmetric}\ \}\ ,
\end{equation}
where the symmetry in question is symmetry under permutation of coordinates.  Note that the subspace of symmetric functions is invariant under
$e^{tL_{N,E}}$. It is the spectral gap in this symmetric sector that is relevant  to the Boltzmann equation; see \cite{K56,K59}. 

When attempting to determine $\Delta_{N,E}$, a significant difference between the cases $\gamma =0$ and $\gamma > 0$ lies in the
following observation: For each $m\in \N$, the space $\mathcal {P}_{m,N}$ consisting of polynomial functions of $v_1,\dots,v_N$  of degree at most 
$m$ in each $v_j$ is invariant under $\LN$, as follows easily from the definition of $\LN$. 

Thus, the eigenfunctions of $\LN$ will lie in these subspaces, which are finite dimensional. While there is no montonicity argument that provides an {\it a priori} guarantee that the eigenfunction furnishing the spectral gap will be a {\em low degree} polynomial, it is natural to guess 
that this is the case.  Simple computation then show that for all $N\geq 2$,
\begin{equation}\label{gapfu}
f_0(v) = \sum_{j=1}^N \varphi_0(v_j)  \qquad {\rm where} \qquad \varphi_0(w) = w^4 - \frac{3N}{N+2}
\end{equation}
is an eigenfunction of $\LN$ with eigenvalue 
\begin{equation}\label{gapva}
\frac12 \frac{N+2}{N-1}\ .
\end{equation}
 It is natural to guess that (\ref{gapva}) is the spectral gap $\Delta_N$, and that (\ref{gapfu}) is the gap eigenfunction. (There are a few other symmetric low-degree candidates one might try, but none of them does as well.) 

We emphasize that even for $\gamma= 0$, not all of the eigenfunctions of $\LN$ have this simple structure, though, the gap eigenfunction does, as the proof in \cite{CCL00} shows.
It is because of this fact that for $\gamma=0$, $\lim_{N\to\infty}\Delta_N$ is exactly equal to the spectral gap of the linearized
Boltzmann equation, as we discuss below.  

For $\gamma>0$, there are no polynomial eigenfunctions of $\LN$, nor are their eigenfunction of the form $\sum_{j=1}^N \varphi(v_j)$.
Nonetheless, there are good reasons to expect that the gap eigenfunction for $\gamma> 0$ must be {\em approximately of this form}
for large $N$,  as we explain below.

Therefore, define $\mathcal{A}_{E}$ to be the subspace of $\mathcal{H}_{E}$ consisting of functions $f$ such that 
$$f(v) = \sum_{j=1}^N \varphi(v_j)\ ,$$
and define
\begin{equation}\label{gapdef2}
\widehat \Delta_{N,E} := \inf\{ \E(f,f) \ :\ \|f\|_2 = 1\ ,\  \langle f,1\rangle =0\ ,\ f \in \mathcal{A}_E\ \}\ .
\end{equation}
The subspace  $\mathcal{A}_E$ is {\em not} invariant under $\LN$, and $\widehat \Delta_{N,E}$
is not an eigenvalue of $\LN$. Nonetheless, if it is indeed true that for large $N$ the gap eigenfunction is {\em close} to an element of 
$\mathcal{A}_{E}$, one would expect $\widehat \Delta_{N,E}$ to be only slightly larger than $\Delta_{N,E}$ for large $N$. 

As we shall show, there is indeed a very close connection between $\widehat \Delta_{N,E}$ and $\Delta_{N,E}$,
and that solving  the more manageable problem of proving a lower bound on the former quantity leads to a lower bound on the latter quantity.

For these reasons, we are interested in how both $\Delta_{N,E}$ and $\widehat \Delta_{N,E}$ vary with $N$ and $E$ as $\gamma$ is held fixed. The dependence on $E$ is a simple matter of scaling.

\begin{lm}\label{scale} For all $N>0$, $\gamma\in [0,1]$, and $E,E'>0$,
\begin{equation}\label{scaleB}
\Delta_{N,E} = \left(\frac{E}{E'}\right)^\gamma \Delta_{N,E'}\ ,
\end{equation}
and the same relation holds for $\widehat \Delta_{N,E}$.
\end{lm}

\begin{proof}
Suppose $f$ is any measurable function on ${\mathcal S}_{N,E'} = S^{N-1}(E')$. Define $Sf$ by
\begin{equation}\label{scaleC}
Sf( v) = f\left( \sqrt{\frac{E'}{E}} v\right)\ .
\end{equation}
Then, with the uniform probability measure on both spheres, it is clear that $f \mapsto Sf$
is unitary from $L^2({\mathcal S}_{N,E'})$ to $L^2(\STE)$, and that
\begin{equation}\label{scaling}
N{\ncht}^{-1}\sum_{i<j}\int_{{\mathcal S}_{N,E'}}(v_i^2+v_j^2)^\gamma (f - [f]^{(i,j)})^2\dd \sigma = 
\left(\frac{E'}{E}\right)^\gamma \E(Sf,Sf) \ .
\end{equation}
\end{proof}
Note that in the Maxwellian case ($\gamma = 0$), the radius of the sphere is immaterial.  On account of this lemma, we simplify our notation: We shall write
$\Delta_N$ in place of $\Delta_{N,1}$; that is, the spectral gap for $N$ particles with unit energy per particle.  Likewise, we shall write ${\mathcal E}_{N}$
to denote  ${\mathcal E}_{N,1}$, $\mathcal {H}$ to denote  $\mathcal {H}_1$, and  $\mathcal {A}$ to denote  $\mathcal {A}_1$.

The dependence of $\Delta_N$ on $N$ is not so simple to determine. The original conjecture of Kac was that for $\gamma =0$,
\begin{equation}\label{kcon}
\liminf_{N\to\infty}\Delta_{N} > 0\ .
\end{equation}
This was proved in \cite{J01}, with the exact value of $\Delta_N$ obtained in \cite{CCL00}.  Again for $\gamma =0$, the Kac conjecture  for three-dimensional momentum and energy
conserving collisions  \cite{K59} was proved in \cite{CCL03}, and with the exact value of $\Delta_N$ in this three-dimensional case in \cite{CGL}.

It is evident that for each $N$, $\widehat{\Delta}_N \ge \Delta_N$, and almost as evident that this inequality is strict. Nonetheless, the weaker conjecture that
\begin{equation}\label{kcon2}
\liminf_{N\to\infty}\widehat{\Delta}_{N} > 0\ 
\end{equation}
will be shown to be a significant stepping-stone towards the proof of (\ref{kcon}). 

\begin{thm}\label{main} For all $N\geq 3$, 
\begin{equation}\label{main1A}
\widehat{\Delta}_N \geq  \left(1 - \frac{A_N}{N^2}\right)\widehat{\Delta}_{N-1}
\end{equation}
where
\begin{equation}\label{an}
A_N = \frac{p(N) + \gamma q(N)}{r(N)}\ 
\end{equation}
with
\begin{eqnarray}
p(N) &=& 5N^7 + 31N^6 + 15N^5 + 131N^4 +256N^3 -102N^2\nonumber\\
q(N) &=& 5N^7 - 5N^6 -87N^5 -211N^4 -164N^3 + 78N^2\nonumber\\
r(N) &=& (N^2+4N -12)(N-1)^3(N+1)(N-2)\ .\label{pqrn}
\end{eqnarray}
Moreover for all $N_0$ such that  $A_N/N^2 < 1$ for all $N \geq N_0$, which is true for
all sufficiently large $N_0$, 
\begin{equation}\label{expo}
\widehat{\Delta}_N \geq   N_0^{\gamma -1} \left(\prod_{j=3}^{N_0} 
 \left[1 -   \frac{4j+1}{(j-1)^2(j+1)}\right]\right) \prod_{k \geq N_0}^N \left(1 - \frac{A_k}{k^2}\right) \ ,
\end{equation}
 and in particular, (\ref{kcon2}) is true, with a computable lower bound. 
\end{thm} 

As indicated by (\ref{main1A}), our proof of (\ref{kcon2}) proceeds by  induction on the number of particles, 
as in our previous works concerning the $\gamma= 0$ case.
Toward this end, we note that it is easy to compute $\Delta_2$.

\begin{lm}\label{Delta2}
For all $0 \leq \gamma \leq 1$,  
$$\widehat{\Delta}_2 = \Delta_2 = 2^{\gamma+1}\ .$$
\end{lm}

\begin{proof} This is a simple calculation.
\end{proof}

Some remarks on the connection between (\ref{expo}) and (\ref{kcon2}) are in order. 
The estimate (\ref{expo}) is only meaningful in case $A_N/N^2 < 1$. However,
since
$\lim_{N\to \infty}A_N =: 5(1+\gamma)$  exists, there exists an $N_0$ so that  $A_N/N^2 < 1$ for all $N \geq N_0$, and for such $N_0$,
$$\prod_{j \geq N_0}^\infty \left(1 - \frac{A_N}{N^2}\right) > 0\ .$$

Thus, (\ref{expo}) reduces the proof of (\ref{kcon2}) to showing that $\widehat{\Delta}_{N_0} > 0$ for {\em fixed} $N_0$. This is relatively easy:
By a fairly direct adaptation of the method used in out earlier work on $\gamma =0$, we prove in Theorem~\ref{uniform}
that for each $N_0\geq 3$, 
$$\widehat{\Delta}_{N_0} \geq \Delta_{N_0} \geq 4 N_0^{\gamma -1} \left(\prod_{j=3}^{N_0}  \left[1 -   \frac{4j+1}{(j-1)^2(j+1)}\right]\right)>0 \ .$$
Altogether, we obtain (\ref{kcon2}).

 Moreover, simple computations show that for $\gamma =1/2$, the ``hard sphere'' case, $A_N/N^2 \leq 0.542$ for all $N\geq 6$.
 Hence we may take $N_0 =6$ in this case. As we shall see, (\ref{expo}) leads to good numerical lower bounds on $\widehat{\Delta}_N$. 
 
 As we shall see, once Theorem~\ref{main} has been proved, it is relatively simple to get a lower bound on $\Delta_N$ that is uniform in $N$. 
 The key to this is a decomposition $f =g+h$ of an arbitrary admissible trial function $f$ in the variational principle for $\Delta_N$ into pieces
 $g\in \mathcal{A}_N$ and $h$ that is ``harmless''. Thus, we shall prove:

 \begin{thm}  \label{mainB} For all $N\geq 3$, 
\begin{equation}\label{main1}
\Delta_N \geq   \left(1 - \frac{A_N+ C_N}{N^2}\right)\Delta_{N-1}
\end{equation}
where
$A_N$ is given by \eqref{an} and where $C_N$ is given by 
 \begin{equation}\label{crt3}
 C_N := \sqrt{15}\frac{1-\gamma}{(N-1)^2} N^{5/2}
  \left[\frac{2}{N-1} +\frac{8N}{(N-2)(N-4)^2}\right] ^{1/2} \left(1 - \frac{15}{ {(N+1)}(N+3)}\right)^{-1/2}\ .
   \end{equation}
Moreover, for all $N_0$ such that $A_N +C_N < N^2$ for all $N>N_0$, which is true for all sufficiently large $N_0$, 
$$\Delta_N \geq  4 N_0^{\gamma -1} \left(\prod_{j=3}^{N_0} 
 \left[1 -   \frac{4j+1}{(j-1)^2(j+1)}\right]\right) \prod_{j \geq N_0}^N \left(1 - \frac{A_N+ C_N}{N^2}\right) > 0\ .$$
 In particular, 
$$
\liminf_{N\to\infty}\Delta_N > 0\ .
$$
\end{thm}

\subsection{Application to the Kac-Boltzmann equation}

Kac's original motivation for introducing the Kac model was to study the {\em Kac-Boltzmann} equation, which is a caricature of the spatially homogeneous
Boltzmann equation for one dimensional velocities:
\if false
As we shall explain at the end of the paper, the methods employed here apply also to three-dimensional
velocities and collisions that preserve energy and momentum, leading to the standard spatially homogeneous Boltzmann equation. However, \fi 
\begin{equation}\label{KNE}
\frac{\partial}{\partial t} f(v,t) =  \Q(f,f)(v,t)
\end{equation}
where $f(v,t)$ gives the probability that a randomly selected molecule will have velocity $v$ at time $t$ and where
the {\em collision kernel}  $Q(f,f)$ for the Kac-Boltzmann equation is given by
\begin{equation}\label{kbk}
\Q(f,f) = \frac{1}{\pi}\int_{\R}\int_{-\pi}^{\pi} (v^2+w^2)^\gamma\left[
f(v_*)f(w_*) - f(v)f(w)\right]\dd \theta \dd w \ .
\end{equation}
with
\begin{equation}\label{kbkb}
v_* = v\cos\theta + w\sin\theta \qquad{\rm and}\qquad w_* = -v\sin\theta+ w\cos\theta\ .
\end{equation}

 Kac proved that there is a close relation between the Kac walk, and the Kac-Boltzmann equation through his
notion of {\em propagation of chaos}. Our next theorem shows another aspect of this close relation. 
It concerns the {\em spectral gap of the linearized collision operator}. 

To explain this, we first observe  that $\Q(f,f) = 0$ if and only if $f$ is a centered Maxwellian density; i.e.,
\begin{equation}\label{max}
f(v) = \frac{1}{\sqrt{2\pi \Theta}}e^{-v^2/2\Theta}
\end{equation}
for some $\Theta>0$, where $\Theta$ is the second moment of the probability density 
in (\ref{max}). By an appropriate choice of units, we may suppose that $\Theta =1$, and we define
the {\em unit Maxwellian density} $M$ by
\begin{equation}\label{max2}
M(v) = \frac{1}{\sqrt{2\pi }}e^{-v^2/2}
\end{equation}

The {\em linearized Kac-Boltzmann operator} ${\mathcal L}$ is obtained by considering
small perturbations of $M$ of the form 
\begin{equation}\label{max3}
f(v) = M(v)[1 + h(v)]
\end{equation}
where $h$ satisfies
\begin{equation}\label{max4}
\int_\R  h(v)M(v)\dd v = \int_\R v^2h(v)M(v)\dd v = 0\ ,
\end{equation}
and $h(v) \geq -1$ for all $v$. In this case, (\ref{max3}) defines a probability density
with unit second moment, as does $M$. Thinking of $h$ as small, one finds that
with $f$ given by  (\ref{max3}),
\begin{equation}\label{kbk2}
\Q(f,f) = \frac{1}{\pi}\int_{\R}\int_{-\pi}^{\pi} (v^2+w^2)^\gamma [
h(v_*) + h(w_*) - h(v) -h(w)]M(v)M(w)
\dd \theta \dd w + {\mathcal O}(h^2)\ .
\end{equation}
\if false
Since ${\displaystyle  \int_{-\pi}^{\pi} h(v_*) \dd \theta =  \int_{-\pi}^{\pi} h(w_*) \dd \theta}$
and since ${\displaystyle  \int_\R  h(v)M(v)\dd v = 0}$, if we define
\fi
We define
\begin{equation}\label{lin}
{\mathcal L} h(v) =  \frac{1}{\pi}\int_{\R}\int_{-\pi}^{\pi} (v^2+w^2)^\gamma [
h(v_*)  + h(w_*) - h(v) - h(w)]M(w)
\dd \theta \dd w\ ,
\end{equation}
then we can rewrite (\ref{kbk2}) as
\begin{equation}\label{kbk2B}
\Q(f,f) = M(v)  {\mathcal L} h(v) +  {\mathcal O}(h^2)\ .
\end{equation}
It is easily checked that ${\mathcal L}$ is self adjoint on $L^2(\R,M(v)\dd v)$.  We compute
\begin{multline}\label{lin2}
\langle h,{\mathcal L} h\rangle_{L^2(\R,M(v)\dd v)}
 =\\  \frac{1}{\pi}\int_{\R}\int_{\R}\int_{-\pi}^{\pi} (v^2+w^2)^\gamma [
h(v)h(v_*) + h(v)h(w_*) - h^2(v) - h(v)h(w)]M(v)M(w)
\dd \theta \dd w\dd v\ ,
\end{multline}
It is easy to see that the functions $h(v) =1$ and $h(v) = v^2$ are in the nullspace of ${\mathcal L}$.

We then define $\Lambda$, the {\em spectral gap}  of ${\cal L}$,  by
\begin{equation}\label{lin3}
\Lambda = \inf\left\{- \frac{\langle h,{\mathcal L} h\rangle_{L^2(\R,M(v)\dd v)} }{\|h\|^2_{L^2(\R,M(v)\dd v)} }\ :\   \langle h,1\rangle_{L^2(\R,M(v)\dd v)}  =  \langle h,v^2\rangle_{L^2(\R,M(v)\dd v)} = 0\ .\right\}
\end{equation}

Our terminology would suggest that the null space of ${\mathcal L}$ is spanned by $1$ and $v^2$,
and that ${\mathcal L}$ is negative semi-definite.  This turns out to be the case. It is easy to see this for
$\gamma = 0$, since then ${\mathcal L}$ can be expressed as an average over Mehler kernels, and the spectrum of ${\mathcal L}$ computed exactly. For other values of $\gamma$, direct computation 
is not possible.

Our next theorem, together with our analysis of the Kac master equation,  shows that not only  is ${\mathcal L}$  negative semidefinite, but that $\Lambda$ is 
strictly positive, and moreover, 
provides  a computable lower bound on $\Lambda$.

\begin{thm}\label{apl}   For all $0 \leq \gamma < 1$, 
\begin{equation}\label{lin4}
\Lambda \geq \limsup_{N\to\infty}\widehat \Delta_N\ .
\end{equation}
\end{thm}

If one is mainly interested in lower bounds on $\Lambda$, Theorem~\ref{apl} shows that 
 $\widehat{\Delta}_N$ is the main quantity of interest associated to the Kac walk, moreso than ${\Delta}_N$
 in this particular regard. 
 
\begin{thm} For $\gamma = 1/2$,
$$\Lambda \geq 0.0263\ .$$
\end{thm}

\begin{proof}  By Theorem~\ref{apl}, it suffices to produce a uniform lower bound on $\widehat{\Delta}_N$ for $\gamma =1/2$. 
Simple calculations show that $A_6/6^2 \leq 0.54$, and is below this value for all $k > 6$. Thus, we may take $N_0= 6$ in 
(\ref{expo}).  However, experimentation shows that $N_0 = 10$ yields the optimal result. We  evaluate
$$\prod_{k=11}^{10^6} (1 - a_k/k^2) = 0.5067...$$
numerically, and bound the remainder by an integral comparison. Combining the terms, we obtain $\widehat{\Delta}_N \geq 0.0263$ 
for all $N\geq 2$. 
\end{proof} 

\begin{remark} If carry out the analogous estimate for $\gamma=0$, we obtain the bound 
$$\Lambda \geq 0.0592$$
in this case. For $\gamma= 0$, the actual value is $\Lambda =1/2$, as is easily computed and well known, see  \cite{CCL03} for discussion.
Thus, the lower bound we obtain by this procedure  falls short of the actual value by less than one order of magnitude. We may expect the
hard sphere bound to be comparably accurate. 
\end{remark}

 The rest of the paper is organized as follows: 
 In Section 2 we  set up an induction scheme, and use it to prove a
 simple estimate on the gap that  is uniform in $N$  if the
 collision rate is proportional to $E_{i,j}^\gamma$ for $\gamma =0$ and $\gamma=1$, but only in these cases. The  result for $\gamma =1$ is foreshadowed by the work of Villani \cite{V} on entropy production in the super-hard sphere case, though it is not a consequence of his work. 
Analysis of this proof highlights the difficulties to be overcome for $0< \gamma < 1$.
 
 In Section 3 we show how the induction scheme may be adapted to prove a lower bound on $\widehat{\Delta}_N$ that is independent of $N$. 
In Section 4,  we introduce a decomposition of trial functions to be used in the variational formula for the spectral gap ${\Delta}_N$, and show how this decomposition
reduces the problem of bounding  ${\Delta}_N$ to that of bounding  $\widehat{\Delta}_N$.   In Section 5 we
prove Theorem~\ref{apl}, and show how  Theorem \ref{main} leads to
 an explicit estimate for the linearized Kac-Boltzmann equation. Certain technical results concerning correlations on the sphere, which may have other applications,
 are proved in Section 6, and a brief Section 7 explain a method for evaluating certain infinite products that we encounter here.

\section{The induction}

Fix an admissible trial function $f$ in the variational formula for $\Delta_N$. To do the induction, define the {\em conditional Dirichlet form} $\E(f,f|v_k)$
given by
\begin{equation}\label{cdir}
{\mathcal E}_N (f,f|v_k) = 
{(N-1)}{\nmcht}^{-1}\sum_{i<j; i,j\neq k}\int_{S^{N-2}(\sqrt{N -v_k^2})}(v_i^2+v_j^2)^\gamma (f - [f]^{(i,j)})^2\dd \sigma\ ,
\end{equation}
where the integration on the right is only over the ``slices'' of $S^{N-1}(\sqrt{N})$ at constant values of $v_k$, so that
the result is still a non-trivial function of $v_k$. 
Let $\dd \nu_N(v_k)$ be the marginal distribution induced on $[-\sqrt{N},\sqrt{N}]$ by the map
$ v \mapsto v_k$ and the uniform probability measure $\dd\sigma$ on 
$S^{N-1}(\sqrt{N})$. 

One easily checks that
\begin{equation}\label{rec}
{\mathcal E}_N(f,f) = \frac{N}{N-1}
\left(\frac{1}{N}\sum_{k=1}^N  \int_{-\sqrt{N}}^{\sqrt{N}} {\mathcal E}_N(f,f|v_k) \dd \nu_N(v_k)\right)\ .
\end{equation}

Furthermore, 
$${\mathcal E}_N(f,f|v_k)= {\mathcal E}_{N-1,\sqrt{N-v_k^2}}(g,g)$$
where $g$ is the restriction of  $f$ to the 
slice of $S^{N-1}(\sqrt{N})$ at constant  $v_k$. If $g$ were orthogonal to the constants in $L^2(S^{N-2}(\sqrt{N-v_k^2}))$, we could estimate the right hand side
in terms of $\Delta_{N-1,\sqrt{N-v_k^2}}$. By Lemma~\ref{scale}, 
\begin{equation}\label{scale2}
 \Delta_{N-1,\sqrt{N-v_k^2}} = \left(\frac{N - v_k^2}{N-1}\right)^\gamma \Delta_{N-1}\ .
 \end{equation}
Combining this with (\ref{rec}) would yield an estimate for $\Delta_N$ in terms of $\Delta_{N-1}$.

However, even if $f$ is orthogonal to the constants in $L^2(\STE)$, it need not be the case that $g$ 
is orthogonal to the  constants  on its slice.
To correct for this, we need to add and subtract the average of $f$ over these slices. The average of $f$ over the $k$th   slice is
the average of  $f(Rv)$ over all rotations that fix the $k$th axis, which in turn is the orthogonal projection in $L^2(\STE)$ onto the subspace of functions depending {\em only} on the coordinate $v_k$.  For each $k=1,\dots, N$, we define $P_k$ to be this orthogonal projection. 
Note that since $P_k f$ depends only on $v_k$, 
$${\mathcal E}_N(f,f|v_k)  = {\mathcal E}_N(f- P_kf,f-P_kf |v_k) \ .$$
Then by the definition of the spectral gap and scaling relation (\ref{scale2}),
\begin{equation*}
 \mathcal{E}_N(f- P_kf,f-P_kf |v_k) \geq \left(\frac{N - v_k^2}{N-1}\right)^\gamma\Delta_{N-1} \Vert f - P_k f\Vert^2_{L^2(S^{N-2}(\sqrt{N -v_k^2}))}\ .
\end{equation*}
 
Finally, we have the obvious identity
$$ \Vert f - P_k f\Vert^2_{L^2(S^{N-2}(\sqrt{N -v_k^2}))} =
 \Vert f \Vert^2_{L^2(S^{N-2}(\sqrt{N -v_k^2}))} -
  \Vert  P_k f\Vert^2_{L^2(S^{N-2}(\sqrt{N -v_k^2}))}\ .$$

Altogether, going back to (\ref{rec}), one has
\begin{equation}\label{rec2}
 \mathcal{E}_N (f,f) \geq \frac{N}{N-1}\Delta_{N-1}
\left(\frac{1}{N}\sum_{k=1}^N  \int_{S^{N-1}(\sqrt{N})}\left(\frac{N - v_k^2}{N-1}\right)^\gamma
\left[f^2 - |P_kf|^2\right]  \dd \sigma\right)\ .
\end{equation}
To summarize conveniently our conclusions, we define the quadratic form $\mathcal{F}_N(f,f) $ on $L^2(S^{N-1}(\sqrt{N}))$
as follows:

\begin{defi}\label{fformdef}
 \begin{equation}\label{find2}
\mathcal{F}_N(f,f) := 
  \frac{1}{N} \sum_{k=1}^N\left[   \int_{S^{N-1}(\sqrt{N})} w^{(\gamma)}(v_k)[f - P_kf]^2\dd \sigma
 \right]\ 
\end{equation}
where
 \begin{equation}\label{find3}
 w^{(\gamma)}(v) = \left(\frac{N - v^2}{N-1}\right)^\gamma\ .
 \end{equation}
\end{defi}
We have proved:
 \begin{thm}\label{thm1} For all symmetric $f\in L^2(S^{N-1}(\sqrt{N}))$ with $\|f\|^2_2 =1$ and with $f$ orthogonal to the constants,
 \begin{equation}\label{ind1}
 \mathcal{E}_N (f,f) \geq \left[\frac{N}{N-1}\Delta_{N-1}\right] \mathcal{F}_N(f,f) \ .
\end{equation}
\end{thm}

We therefore define
\begin{defi}
\begin{equation}\label{gapdefinition}
\Gamma_{N} := \inf\{ \mathcal{F}(f,f) \ :\ \|f\|_2 = 1\ ,\  \langle f,1\rangle =0\ ,\ f\ {\rm symmetric}\ \}\ ,
\end{equation}
\end{defi}

Combining this definition with (\ref{gapdef}), Theorem~\ref{thm1} yields
\begin{equation}\label{red1}
\Delta_N \geq   \Delta_{N-1} \frac{N}{N-1} \Gamma_N
\end{equation}

If we can succeed in showing that for all $N\geq 3$
\begin{equation}\label{red2}
\Gamma_N \geq \frac{N-1}{N}(1-a_N)
\end{equation}
where 
\begin{equation}\label{red3}
0\leq a_N \leq 1 \quad{\rm for\ all}\ N\qquad{\rm and}\qquad \sum_{N=3}^\infty a_N < \infty\ ,
\end{equation}
it will follow that for all $N$,
\begin{equation}\label{prod}
\Delta_N \geq \Delta_2 \prod_{j=3}^\infty (1-a_j) > 0\ ,
\end{equation}
providing the bound we seek.

\if false
\begin{remark} At this point we have made no use of the symmetry hypothesis, and the same result is valid for the spectral gap on the whole
space $\mathcal{H}$. 
\end{remark}
\fi

\begin{remark}\label{reduced} More significantly, notice that if $f\in \mathcal{A}_N$; i.e., $f(v) = \sum_{j=}^N\varphi(v_j)$
for some function $\varphi$, then the restriction of $f$ to any slice on which $v_k$ is constant belongs to   $\mathcal{A}_{N-1}$
Therefore, if we define $\widehat{\Gamma}_N$ by
\begin{equation}\label{redinfA}
\widehat{\Gamma}_N := \inf\left\{ \mathcal{F}_N(f,f)  \ \bigg|\ \|f\|_2 =1\ ,\ \langle f,1\rangle = 0\  ,\ f \in \mathcal{A}_N\right\}\ ,
\end{equation}
we have 
\begin{equation}\label{redB}
\widehat{\Delta}_N \geq   \widehat{\Delta}_{N-1} \frac{N}{N-1} \widehat{\Gamma}_N\ ,
\end{equation}
and are thus motivated to seek bounds of the type (\ref{red2}) and (\ref{red3}) for $\widehat{\Gamma}_N$. 
\end{remark}

We now turn to the task of proving such bounds.

\if false

we define the self-adjoint operator $P^{(\gamma)}$ by
\begin{equation}\label{pgdef}
P^{(\gamma)} = \frac{1}{N}\sum_{k=1}^N \left(\frac{N - v_k^2}{N-1}\right)^\gamma P_k\ .
\end{equation}
Notice that for each $k$, both $P_k$ and the multiplication operator 
${\displaystyle \left(\frac{N - v_k^2}{N-1}\right)^\gamma}$ are commuting and self adjoint, so that
$P^{(\gamma)}$ itself is indeed self adjoint, and even non-negative. 
Also, since each $P_k$ is a projection, we have
\begin{equation}\label{pg2}
\frac{1}{N}\sum_{k=1}^N \int_{S^{N-1}(\sqrt{N})}\left(\frac{N - v_k^2}{N-1}\right)^\gamma
|P_kf|^2 \dd \sigma
 = \langle f, P^{(\gamma)} f\rangle_{L^2(S^{N-1}(\sqrt{N}))}\ .
 \end{equation}
 
Defining the function $W^{(\gamma)}$ by
 \begin{equation}\label{wgdef}
W^{(\gamma)} = \frac{1}{N}\sum_{k=1}^N \left(\frac{N - v_k^2}{N-1}\right)^\gamma \ ,
\end{equation}
we can summarize our conclusions:

As a consequence, 
\begin{equation}\label{red}
\Delta_N \geq   \Delta_{N-1} \frac{N}{N-1} \Xi_N
\end{equation}
where
\begin{equation}\label{redinf}
\Xi_N := \inf\left\{  \int_{S^{N-1}(\sqrt{N})} W^{(\gamma)} f^2\dd \sigma
  - \langle f, P^{(\gamma)} f\rangle_{L^2(S^{N-1}(\sqrt{N}))} \ \bigg|\ \|f\|_2 =1\ ,\ \langle f,1\rangle = 0\  ,\ f \ {\rm symmetric}\right\}\ .
\end{equation}
Iterating (\ref{red}), 
\begin{equation}\label{prod}
\Delta_N \geq \prod_{j=3}^N \left(\frac{j}{j-1}\Xi_j\right)\Delta_2   = \frac{N}{2}\left(\prod_{j=3}^N \Xi_j\right) \Delta_2\ .
\end{equation}
The right hand side is uniformly bounded and bounded away form zero if and only if
$$\Xi_j = 1 - \frac{1}{j}  + a_j$$
where $\sum_{j=3}^\infty|a_j| < \infty$.  The rest of the paper is devoted to establishing this as a fact. Note that it is important that the constant in the $1/j$
term is {\em exactly} $1$.  In the next section we collect relevant bounds.

\fi

\subsection{Some simple but useful bounds}

For our first approach to bounding $\mathcal{F}_N$ from below, we first rewrite this quantity as a difference of two terms. 
Toward this end, 
we define the self-adjoint operator $P^{(\gamma)}$ by
\begin{equation}\label{pgdef}
P^{(\gamma)} = \frac{1}{N}\sum_{k=1}^N \left(\frac{N - v_k^2}{N-1}\right)^\gamma P_k\ .
\end{equation}
For each $k$, both $P_k$ and the multiplication operator 
${\displaystyle \left(\frac{N - v_k^2}{N-1}\right)^\gamma}$ are commuting and self adjoint, so that
$P^{(\gamma)}$ itself is self-adjoint, and even non-negative. 
Since each $P_k$ is a projection, we have
\begin{equation}\label{pg2}
\frac{1}{N}\sum_{k=1}^N \int_{S^{N-1}(\sqrt{N})}\left(\frac{N - v_k^2}{N-1}\right)^\gamma
|P_kf|^2 \dd \sigma
 = \langle f, P^{(\gamma)} f\rangle_{L^2(S^{N-1}(\sqrt{N}))}\ .
 \end{equation}
 Define the function $W^{(\gamma)}$ by
 \begin{equation}\label{wgdef}
W^{(\gamma)} = \frac{1}{N}\sum_{k=1}^N  w^{(\gamma)}(v_k) = \frac{1}{N}\sum_{k=1}^N \left(\frac{N - v_k^2}{N-1}\right)^\gamma \ ,
\end{equation}
we may rewrite $\mathcal{F}_N$ as
\begin{equation}\label{difform}
\mathcal{F}_N(f,f)   = \int_{S^{N-1}(\sqrt{N})} W^{(\gamma)} f^2\dd \sigma
  - \langle f, P^{(\gamma)} f\rangle\ .
  \end{equation}
From  an upper bound of $P^{(\gamma)}$, and  a lower bound on $W^{(\gamma)}$, we can deduce  a lower bound on $\mathcal{F}_N$. 
It is easy to deduce a bound on $ \langle f, P^{(\gamma)}f\rangle$ for $f$ orthogonal to the constants
from the calculation of the spectral gap  of the operator $P^{(0)}$ that was made in \cite{CCL00,CCL03}.

\begin{lm}\label{proj} For all $N$, all $0 \leq \gamma \leq 1$, and all $f \in L^2(S^{N-1}(\sqrt{N}))$ orthogonal to the constants, 
\begin{equation}\label{proj2}
\langle f, P^{\gamma}f\rangle  \leq \left(\frac{N}{N-1}\right)^\gamma \mu_N \|f\|_2^2 \ 
\end{equation}
where
\begin{equation}\label{pgap}
\mu_N = \frac{1}{N} + \frac{3}{N(N+1)}\ .
\end{equation}
\end{lm}
\medskip
\begin{proof} Using the pointwise upper bound
${\displaystyle \left(\frac{N - v_k^2}{N-1}\right) \leq \frac{N}{N-1}}$, we have
\begin{multline}
\langle f,P^{(\gamma)} f\rangle = \frac{1}{N}\sum_{k=1}^N  \int_{S^{N-1}(\sqrt{N})}\left(\frac{N - v_k^2}{N-1}\right)^\gamma |P_kf|^2\dd \sigma
\leq\\
\frac{1}{N}\sum_{k=1}^N  \int_{S^{N-1}(\sqrt{N})}\left(\frac{N}{N-1}\right)^\gamma |P_kf|^2\dd \sigma
= \left(\frac{N}{N-1}\right)^\gamma \langle f,P^{(0)}f\rangle\ .
\end{multline}
$P^{(0)}$ coincides with the operator $P$ analyzed in \cite{CCL00,CCL03} where it was shown that $\mu_N$ is its second largest eigenvalue after the eigenvalue $1$
corresponding to the constant function.
\end{proof}

\begin{lm}\label{weight} For all $N$, all $0 < \gamma \leq 1$, and for all $ v \in S^{N-1}(\sqrt{N})$,
\begin{equation}\label{wlb1}
\left(\frac{N-1}{N}\right)^{1-\gamma} \ \leq \ W^{(\gamma)}( v)\  \leq\  1\ .
\end{equation}
Furthermore,  for all $ v \in S^{N-1}(\sqrt{N})$,
\begin{equation}\label{wlb2}
W^{(0)}( v)\  = \  W^{(1)}( v)  \  = \  1
\end{equation}
\end{lm}

\begin{proof} Since
${\displaystyle \frac{1}{N}\sum_{k=1}^Nv_k^2 =1}$,
${\displaystyle  \frac{1}{N}\sum_{k=1}^N    \left(\frac{N - v_k^2}{N-1}\right) = 1}$,
Jensen's inequality yields
$$\frac{1}{N}\sum_{k=1}^N    \left(\frac{N - v_k^2}{N-1}\right)^\gamma \leq 
\left(\frac{1}{N}\sum_{k=1}^N\frac{N - v_k^2}{N-1}\right)^\gamma = 1\ .$$

To prove the lower bound, note that the minimum of ${\displaystyle  v \mapsto 
\frac{1}{N}\sum_{k=1}^N    \left(\frac{N - v_k^2}{N-1}\right)^\gamma}$ is the same as the minimum of the function
\begin{equation} \label{wlb4}
{\displaystyle 
(x_1,\dots,x_N) \mapsto \frac{1}{N}\sum_{k=1}^N    \left(\frac{N - x_k}{N-1}\right)^\gamma
}
\end{equation}
on the set of $(x_1,\dots,x_N)$ satisfying
\begin{equation}\label{wlb5}
x_j \ge 0 \quad {\rm for\ all}\quad j =1,\dots,N \qquad{\rm and}\qquad \sum_{j=1}^Nx_j = N\ .
\end{equation}
Since the function in (\ref{wlb4}) is concave, the minimum occurs at an extreme point of the domain, and by symmetry, all extreme points yield the same value. Thus the minimum occurs at
$$(x_1,\dots,x_N)  = (N,0,\dots,0),$$
and the minimum value is $((N-1)/N)^{1-\gamma}$. 
The assertions about $\gamma=0$ and $\gamma =1$ are obvious. 
\end{proof}

\begin{remark} A simple calculus and convexity argument shows that
\begin{equation}\label{squeeze}
1 - (1-\gamma)\frac{1}{N-1} \leq   \left(\frac{N-1}{N}\right)^{1-\gamma}  \leq  1 - (1-\gamma)\frac{1}{N} 
\end{equation}
for all $N$ and all $0 \leq\gamma \leq 2$.
Note also  that 
${\displaystyle \lim_{\gamma\to0} \left(\frac{N-1}{N}\right)^{1-\gamma} = \frac{N-1}{N} < 1}$
so that although the lower bound in the lemma is sharp, the lack of uniform convergence means that the limiting bound is not sharp. 
\end{remark}

\subsection{Lower bound on $\Delta_N$ using the uniform bound on $W^{(\gamma)}$\ .}

We now use the lower bound \ref{wlb1} together with Lemma~\ref{proj}    to obtain a lower bound on $\Delta_N$. 
Because $ \ref{wlb1}$ is only sharp for $\gamma>0$, and since $\gamma =0$ has been treated in \cite{CCL00,CCL03}, 
we only consider $\gamma>0$ in the next theorem. 

\begin{thm}\label{uniform} For all $0 < \gamma \leq 1$, and all $N\geq 2$, 
\begin{equation}\label{start}
  \Delta_N \geq 4 N^{\gamma -1} \left(\prod_{j=3}^N  \left[1 -   \frac{4j+1}{(j-1)^2(j+1)}\right]\right)\ .
  \end{equation}
\end{thm}

\begin{proof}
Fix $f$ orthogonal to the constants. By Lemma~\ref{proj} and ~\ref{weight} and (\ref{difform}), 
\begin{eqnarray}
\mathcal{F}_N(f,f) &\geq& \left(\frac{N-1}{N}\right)^{1-\gamma}\| f\|_2^2 - \left(\frac{N}{N-1}\right)^\gamma \mu_N \|f\|_2^2  \nonumber\\
  &=&  \left(\frac{N}{N-1}\right)^{\gamma -2}\left[\frac{N}{N-1}  - \left(\frac{N}{N-1}\right)^2 \mu_N \right] \|f\|_2^2\nonumber\\
  &=& \left(\frac{N}{N-1}\right)^{\gamma -2}\left[ 1 -   \frac{4N+1}{(N-1)^2(N+1)}\right]\|f\|_2^2\nonumber\ .
  \end{eqnarray}
  Therefore,
  $$\Gamma_N \geq  \left(\frac{N}{N-1}\right)^{\gamma -1}\left[ 1 -   \frac{4N+1}{(N-1)^2(N+1)}\right]\ .$$
  Thus by (\ref{prod}),
  \begin{equation*}
  \Delta_N \geq \left(\frac{N}{2}\right)^{\gamma -1} \left(\prod_{j=3}^N  \left[1 -   \frac{4j+1}{(j-1)^2(j+1)}\right]\right)\Delta_2\ .
  \end{equation*}
 \end{proof} 
  
By Theorem~\ref{infpro},
 $$\lim_{N\to\infty} \prod_{j=3}^N  \left[1 -   \frac{4j+1}{(j-1)^2(j+1)}\right] = \frac{3}{\Gamma((5+\sqrt{21})/2) \Gamma((5-\sqrt{21})/2) } \approx 
 0.03881503614\ .$$
 
 For $\gamma=1$, we then have the following result:

\begin{cl}\label{lbone}
For $\gamma =1$ and all $N\ge 2$, we have
\begin{equation}\label{lbone1}
\Delta_N \geq   \frac{12 N^{\gamma-1} }{\Gamma((5+\sqrt{21})/2) \Gamma((5-\sqrt{21})/2) } \approx 0.1552601446 N^{\gamma-1}\ .
\end{equation}
\end{cl} 

Using the lower bound from Lemma~\ref{weight} we cannot obtain a lower bound  on $\Delta_N$ which is uniform in $N$ except
when $\gamma =1$. In the next section we shall obtain a bound on $\Gamma_N$, and hence $\Delta_N$,  that is much sharper {\em for large $N$} that leads
to uniform bounds on $\Delta_N$.

\begin{remark} Nowhere in the proof of Theorem~\ref{uniform} have we made any use of the hypothesis that $f$ be symmetric. In fact, the bound
proved in Theorem~\ref{uniform} is a bound on the spectral gap of $L_{N,1}$ on the whole space, not only the symmetric subspace.  In the next section
we shall make use of the symmetry hypothesis. In the final section, we explain how it may be avoided, but at the cost of a numerically worse bound. 
\end{remark}

\section{Lower bound on $\widehat{\Delta}_N$ and $\Delta_N$ for $0 < \gamma < 1$.}

We show in this section that Theorem~\ref{thm1} can be used as the basis of an inductive approach to bounding $\widehat{\Delta}_N$ from below uniformly in $N$.
The principle behind this is that for admissible  trial  functions in $\mathcal{A}_N$, i.e., admissible 
trial functions of the form $f(v) = \sum_{j=1}^N\varphi(v_j)$, the probability
density $f^2$ cannot be too concentrated in places where the weight $W^{(\gamma)}$
is significantly less than 1, and thus one can improve upon the uniform lower bound on $W^{(\gamma)}$. However,  our actual proof  makes somewhat 
indirect use of this.  

They key to taking advantage of the special form of trial functions $f\in \mathcal{A}_N$  is provided by certain correlation inequalities.

 \subsection{The correlation operators}

 Many of the estimates we shall make to bound $\widehat{\Delta}_N$  would be trivial if it were the case that the functions $\varphi(v_j)$, $j=1,\dots,N$ were independent random variables for the uniform distribution on the
 energy sphere. 
 
 For large $N$, the different velocities $v_j$, are, in fact, nearly independent. This is due to the fact that the unit Gaussian 
 distribution
 $${\rm d}\gamma_N := (2\pi)^{-N/2} e^{-|v|^2/2} {\rm d}^Nv$$
 is very tightly concentrated on a close neighborhood of $S^{N-1}(\sqrt{N})$, and in this sense is very close to the uniform measure
 on $S^{N-1}(\sqrt{N})$, ${\rm d}\sigma$. This is an instance of the {\em equivalence of ensembles} in statistical mechanics --- in this case
 ${\rm d}\sigma$ is the {\em microcanonical ensemble} for a gas of $N$ free particles with unit mean energy   per particle, and 
 ${\rm d}\gamma_N$ is the {\em canonical ensemble} for the same system.  Under the canonical ensemble, the various different velocities {\em are} statistically independent.  
 
 We shall require a quantitative measure of the amount of dependence that there is for each finite $N$. We do this through the {\em correlation operators}, the first and simplest of which we now define.
 
 \begin{defi} [The pair correlation operator $K$]  For any unit vector $u$ on $\R^N$,  the map $v\mapsto v\cdot u$
 maps $S^{N-1}(\sqrt{N})$ onto $[-\sqrt{N},\sqrt{N}]$, and pushes the uniform probability ${\rm d}\sigma$ forward onto a probability measure 
 ${\rm d}\nu_N$ on $[-\sqrt{N},\sqrt{N}]$  that is independent of $u$. 
  We define an operator 
$K$ on $L^2([-1,1],{\rm d}\nu_N)$ by
\begin{eqnarray} \label{Koperator}
\int_{S^{N-1}(\sqrt{N})} \varphi(v_1) \psi(v_2) {\rm d} \sigma &=:&  \int_{S^{N-1}(\sqrt{N})} \varphi(v_1) (K\psi) (v_1) {\rm d} \sigma\nonumber\\
&=& \int_{-\sqrt{N}}^{\sqrt{N}} \varphi(w) K\varphi(w){\rm d}\nu_N(w)  \ .
\end{eqnarray}
\end{defi}
In \cite{CCL03} the eigenvalues $\alpha_k$   of this operator have been computed.
The eigenfunctions are polynomials of degree $k$. All  of the odd polynomials are in the kernel of $K$. Moreover
\begin{equation}\label{mono}
|\alpha_k| > |\alpha_{k+2}|
\end{equation}
and 
\begin{equation}\label{eigval}
\alpha_0 = 1 \ , \ \alpha_2 = - \frac{1}{N-1} \ , \ \alpha_4 = \frac{3}{(N-1)(N+1)} \ , \ \alpha_6 = - \frac{15}{(N-1)(N+1)(N+3)} \ .
\end{equation}
For each even $k$, the eigenspace for the eigenvalue $\alpha_k$ is one-dimensional, and is spanned by an even  polynomial of degree $k$. 

Now let $\varphi$ and $\psi$ be functions on $[-\sqrt{N},\sqrt{N}]$ that are orthogonal to the constants; i.e., the eigenspace of $\alpha_0$. 
Then, by the definition above, and what we have said about the eigenvalues of $K$, since $|\alpha_2| = 1/(N-1)$, 
\begin{eqnarray}
\left| \int \varphi(v_1) \psi(v_2){\rm d}\sigma  -  \left(\int \varphi(v_1) {\rm d}\sigma \right)\left(\int \psi(v_2){\rm d}\sigma \right)\right|
&=& \left|\int_{-\sqrt{N}}^{\sqrt{N}} \varphi(w) K\varphi(w){\rm d}\nu_N(w) \right|\nonumber\\
&\leq& \frac{1}{N-1} \|\varphi\|_2 \|\psi\|_2\ .\nonumber
\end{eqnarray}
where $\|\varphi\|_2^2$ denotes ${\displaystyle \int_{S^{N-1}(\sqrt{N})}|\varphi(v_1)|^2{\rm d}\sigma = 
\int_{[-\sqrt{N},\sqrt{N}]} |\varphi(w)|^2{\rm d}\nu_N(w)}$.   Thus for large $N$, the random variables $\varphi(v_1)$ and $\psi(v_2)$ are almost 
uncorrelated.   

If we know that $\varphi(v_1)$ and $\psi(v_2) $ are not only orthogonal to the constants but are also
orthogonal to $v_1^2$ and $v_2^2$ respectively, then $\varphi$ and $\psi$  are orthogonal to the $\alpha_2$ eigenspace as well as the $\alpha_0$ eigenspace, and  we obtain the stronger bound
$$
\left| \int \varphi(v_1) \psi(v_2){\rm d}\sigma  -  \left(\int \varphi(v_1) {\rm d}\sigma \right)\left(\int \psi(v_2){\rm d}\sigma \right)\right|
\leq  
 \frac{3}{N^2-1} \|\varphi\|_2 \|\psi\|_2\ .
 $$
In this case, we get a much stronger bound on correlations. The following lemma will make this stronger bound available
to us and we shall use it a number of times in what follows.

 \begin{lm}\label{gprop1}
 Let $g\in L^2(S^{N-1}(\sqrt{N}))$ be orthogonal to the constants, and in $\mathcal A_N$, i.e., of the form
 $\sum_{k=1}^N \varphi(v_k)$. The choice of $\varphi$ is not unique, but among the possible choices, there is always one with the property that
  for each $k$, 
 $\varphi(v_k)$ is orthogonal to both $1$ and  $v_k^2$ in  $L^2(S^{N-1}(\sqrt{N}))$.
 \end{lm}

\begin{proof}
Since ${\displaystyle \int_{S^{N-1}(\sqrt{N})}\varphi(v_k)\dd \sigma}$ is independent of $k$,
we have that for each $k$,
$$0 = \int_{S^{N-1}(\sqrt{N})}g( v)\dd \sigma = N\int_{S^{N-1}(\sqrt{N})}\varphi(v_k)\dd \sigma\ ,$$
and so ${\displaystyle \int_{S^{N-1}(\sqrt{N})}\varphi(v_k)\dd \sigma = 0}$.  Next, let $\eta(v_k) := v_k^2 -1$ and define
$$\widetilde \varphi(v_k) = \varphi(v_k) - \left[\left(\int_{S^{N-1}(\sqrt{N})}\eta^2(v_k)\dd \sigma\right)^{-2}
\int_{S^{N-1}(\sqrt{N})}\varphi(v_k)\eta(v_k)\dd \sigma\right]\eta(v_k)\ .$$
By symmetry, the coefficient of $\eta(v_k)$ does not depend on $k$, and then since $\sum_{k=1}^N \eta(v_k) = 0$,
it follows that 
$$\sum_{k=1}^N\widetilde \varphi(v_k) = \sum_{k=1}^N\varphi(v_k) = g(v)\ .$$
By construction,  $\widetilde \varphi(v_k)$ is orthogonal to both $1$ and  $v_k^2$ in  $L^2(S^{N-1}(\sqrt{N}))$.
\end{proof}

Thus, we may assume henceforth that
\begin{equation} \label{orth1}
\int_{S^{N-1}(\sqrt{N})} \varphi(v_1) {\rm d} \sigma = 0 \ ,
\qquad
{\rm and}\qquad
\int_{S^{N-1}(\sqrt{N})} \varphi(v_1) v_1^2 {\rm d} \sigma = 0 \ .
\end{equation}

The orthogonality provided by Lemma~\ref{gprop1} has consequences that are summarized in the next lemma, which shall be used several
times in what follows.

 \begin{lm}\label{gprop2}
 Let $g\in L^2(S^{N-1}(\sqrt{N}))$ be in $\mathcal A_N$ where  
  for each $k$, 
 $\varphi(v_k)$ is orthogonal to both $1$ and  $v_k^2$ in  $L^2(S^{N-1}(\sqrt{N}))$.  Then,
 for each $k$, 
 \begin{equation}\label{gp3}
N \left(1 - \frac{15}{ {(N+1)}(N+3)}\right)\|\varphi(v_k)\|_2^2  \leq \|g\|_2^2 \leq 
N \left(1 + \frac{3}{N+1}\right)\|\varphi(v_k)\|_2^2\ .
 \end{equation}
  \end{lm}
  
  \begin{proof}
Since
$$
\Vert g \Vert^2 = N\Vert \varphi \Vert^2 + N(N-1) \int_{S^{N-1}(\sqrt{N})} \varphi(v_1) (K\varphi)(v_1) {\rm d} \sigma
$$
we find on account of  \eqref{orth1} and the eigenvalues  of  the $K$ operator listed in \eqref{eigval} that
$\alpha_4$ is relevant for the upper bound, and $\alpha_6$  for the lower bound. 
\end{proof}

We shall also make use of higher-order correlation operators, a whole family of which is studied in Section 4. For $N\geq 4$, the operator
$K_{N,2}$, acting on functions $\psi$ on the disk of radius $\sqrt{N}$ in $\R^2$, is defined through the quadratic form
$$\int_{S^{N-1}(\sqrt{N})} \psi(v_1,v_2) [K_{N,2}\psi](v_1,v_2){\rm d}\sigma = 
\int_{S^{N-1}(\sqrt{N})} \psi(v_1,v_2) \psi(v_{N-1},v_N){\rm d}\sigma\ ,$$
in analogy with (\ref{Koperator}). (For $N=3$, $K_{3,2}$ is defined in terms  of projection of functions depending only on $v_2$ and $v_3$ onto the subspace of functions depending only on $v_1$.)

However, most of what follows depends on the properties of the single-particle correlation operator $K$, and we postpone
further analysis of $K_{N,2}$ to Section 4, from which we shall quote results as needed. 

We close this subsection with a few comments on how  correlations  bounds may be used to show that for admissible trial function $f \in \mathcal{A}_N$,
$f^2$ must be largely concentrated on configurations $v$ for which $W^{(\gamma)}$ is very close to 1. 

Suppose $f(v) = \sum_{k=1}^N\varphi(v_k)$ is an admissible trial function, with $\varphi(v_k)$ orthogonal to $1$ and $v_k^2$ in $L^2(S^{N-1}(\sqrt{N})$. Then
$$1 = \int_{S^{N-1}(\sqrt{N})}f^2{\rm d}\sigma =  \sum_{j,k=1}^N\int_ {S^{N-1}(\sqrt{N})}\varphi(v_j) \varphi(v_k){\rm d}\sigma\  .$$
If we make the assumption that the $\varphi(v_k)$ are exactly independent, all terms with $j\ne k$ vanish, and we have that
$$\sum_{k=1}^N \int_{S^{N-1}(\sqrt{N})} \varphi^2(v_k){\rm d}\sigma = 1\ .$$

Again assuming that the coordinate functions are {\em exactly} independent,
$$N =  \int_{S^{N-1}(\sqrt{N})}Nf^2{\rm d}\sigma  =  \sum_{j,k,\ell=1}^N\int_ {S^{N-1}(\sqrt{N})} v_j^2\varphi(v_k)\varphi(v_\ell){\rm d}\sigma =
\sum_{j,k=1}^N\int_ {S^{N-1}(\sqrt{N})} v_j^2\varphi(v_k)^2{\rm d}\sigma\ .$$
This reduces to
${\displaystyle \sum_{k=1}^N\int_ {S^{N-1}(\sqrt{N})} v_k^2\varphi(v_k)^2{\rm d}\sigma = 1}$.
A similar calculation using $N^2 = \sum_{i,j=1}^Nv_i^2v_j^2$ leads to 
${\displaystyle \sum_{k=1}^N\int_ {S^{N-1}(\sqrt{N})} v_k^4\varphi(v_k)^2{\rm d}\sigma = 2N}$.
From here, making further use of the assumed independence, one readily derives
\begin{equation}\label{gby}
\int_ {S^{N-1}(\sqrt{N})} \left[\sum_{k=1}^Nv_k^4\right] f^2{\rm d}\sigma \leq 2N + 3(N-1) \leq 5 N
\end{equation}
for all $N$. 

Recall that $W^{(\gamma)}$ has its minimum at the points of $S^{N-1}(\sqrt{N})$ at which $\sum_{k=1}^Nv_k^4= N^2$. In fact, assuming the bound
(\ref{gby}), it is not hard to show, using Chebychev's inequality,  that
$$\int_ {S^{N-1}(\sqrt{N})}  W^{(\gamma)} f^2 {\rm d}\sigma \geq 1 - \frac{C}{N^{3/2}}$$
for a computable constant $C$. This would resolve the main difficulty we encountered in the previous section. In our actual proof, we employ a somewhat more intricate argument that gives us $\mathcal{O}(1/N^2)$ errors, but we hope this heuristic discussion has explained the utility of the correlation bounds we investigate next.

\subsection{Lower bound on $\widehat{\Delta}_N$.}

We now lay the groundwork for the proof of Theorem~\ref{main}.  We introduce a second approach to bounding $\mathcal{F}_N(f,f)$
from below that will yield more incisive bounds for large values of $N$. The starting point for this approach uses the original formula
for  $\mathcal{F}_N(f,f)$ from Definition~\ref{fformdef}: 
$$
\mathcal{F}_N(f,f) =  \frac{1}{N} \sum_{k=1}^N\left[   \int_{S^{N-1}(\sqrt{N})} w^{(\gamma)}(v_k)[f - P_kf]^2\dd \sigma
 \right]\ 
$$
where ${\displaystyle w^{(\gamma)}(v) = \left(\frac{N - v^2}{N-1}\right)^\gamma}$.
Note that the integrand is positive, and we will exploit some of the cancelations between $f$ and $P_kf$.

\begin{lm}\label{mainL} Let $f$ have the form  $f = \sum_{j=1}^N\varphi(v_j)$ with $\varphi$ orthogonal to $1$ and $v^2$. Then for all $N\geq 3$, 
\begin{equation}\label{main1L}
\mathcal{F}_N(f,f) \geq \frac{N-1}{N}\left(1 - \frac{A_N}{N^2}\right)\|f\|_2^2
\end{equation}
where $A_N$ is given by (\ref{an}), and where $p(N)$, $q(N)$ and $r(N)$ are given by (\ref{pqrn}).
\end{lm} 

Theorem~\ref{main} is an immediate consequence of this Lemma, and previous observation:

\begin{proof}[Proof of Theorem~\ref{main}] It follows from the definition of $\widehat \Gamma_N$ in (\ref{redinfA})  and Lemma~\ref{mainL}
that 
${\displaystyle \widehat \Gamma_N \geq \frac{N-1}{N}\left(1 - \frac{A_N}{N^2}\right)}$. 
It then follows from Remark~\ref{reduced} that the bound in Theorem~\ref{main} is valid.
\end{proof}

To get a close estimate on $\mathcal{F}_N(f,f) $, we need to do a number of exact calculations that can be done with polynomials.
The following lemma will give us the reduction from $w^{(\gamma)}(v)$ to a polynomial weight function.

 \begin{lm}\label{comp1}
 For all $0 < \gamma < 1$ and all $x > -1$,
 \begin{equation}\label{com}
 (1+x)^ \gamma \geq 1 +  \gamma x - (1- \gamma)x^2\ .
 \end{equation}
 Furthermore, for all $\gamma$ such that
 \begin{equation}\label{com2}
 \left(\frac{\gamma}{2}\right)^{1/(2-\gamma)} \geq \frac{1-\gamma}{2-\gamma}\ ,
 \end{equation}
  the function
 \begin{equation}\label{com3}
  x\mapsto  (1- \gamma)x^2 +  (1+x)^ \gamma
 \end{equation}
  is strictly monotone increasing on $(-1,\infty)$.

 \end{lm}
 
 \medskip
 
 \begin{proof}Let $\eta(x)$ be defined by
 $\eta(x) := (1+x)^ \gamma - [ 1 +  \gamma x - (1- \gamma)x^2]$.
 Note that
 $$ \eta''(x) = (1- \gamma)(2 -  \gamma(1+x)^{ \gamma-2})\ .$$
 Thus, $\eta''(x) =0$ has the single solution $x = x_*$ where
 $$ x_*  := ( \gamma/2)^{1/(2- \gamma)} -1\ .$$
 Note that $\eta$ is convex on $(x_*,\infty)$, and concave on $(-1, x_*]$, and also that
  $-1 < x_* < 0$. 
 
 Since $\eta$ is concave on $[-1,x_*]$, 
 $$\min\{ \eta(x) \ :\ -1 \leq x \leq x_*\ \} = \min\{\eta(-1),\eta(x_*)\} = \eta(-1) = 0\ .$$
 Since  $\eta$ is convex  on $(x_*,\infty)$, and this interval contains a point, namely $0$,
 at which $\eta'$ vanishes, the minimum of $\eta$ over this interval is attained at $x=0$, and 
 thus $\eta$ is non-negative on $(x_*,\infty)$ as well as on  $[-1,x_*]$.
 
 For the second part, define $\xi(x) := (1- \gamma)x^2 +  (1+x)^ \gamma$, and note that 
 $\xi''(x) = \eta''(x)$, so that with $x_*$ defined as above, $\xi''(x_*) = 0$. Direct computation shows that
 $\xi'''(x) > 0$ on $(-1,\infty)$, and so $\xi'$ is a strictly convex function on  $(-1,\infty)$.
 It is therefore minimized at $x_*$.  Computing $\xi'(x_*)$, one finds
 $$\xi'(x_*) = 2\left[(2-\gamma)\left(\frac{\gamma}{2}\right)^{1/(2-\gamma)} - (1-\gamma)\right]\ .$$
 This is positive if and only if (\ref{com2}) is satisfied. Thus, $\xi'$ is strictly positive if and only if
 (\ref{com2}) is satisfied.
\end{proof}

Lemma \ref{comp1}  gives us the lower bound
$$w^{(\gamma)}(v) \geq m(v) := 1 +\gamma \left(\frac{1 - v^2}{N-1}\right) 
-(1-\gamma)\left(\frac{1 - v^2}{N-1}\right)^2\ .$$
Note that 
$$m(v) = \left[1-(1-\gamma)\left(\frac{1 - v^2}{N-1}\right) \right]\left[ \left(\frac{1 - v^2}{N-1}\right)  +1\right] > 0\ .$$
Then
$$\mathcal{F}_N(f,f)\geq  \frac{1}{N} \sum_{k=1}^N\left[   \int_{S^{N-1}(\sqrt{N})} m(v_k)[f - P_kf]^2\dd \sigma
 \right]  := \mathcal{G}_N(f,f)\ ,$$
 and so it suffices to prove (\ref{main1L}) with $\mathcal{G}_N(f,f)$ in place of $\mathcal{F}_N(f,f)$.   

\begin{proof}[Proof of Lemma~\ref{mainL}]
Since $f = \sum_{j=1}^N\varphi(v_j)$, 
$$P_kf(v) = \varphi(v_k) + (N-1)K\varphi(v_k)\quad{\rm and\ hence}\quad f(v) - P_k f(v) = \sum_{j\neq k}\varphi(v_j) - (N-1)K\varphi(v_k)\ ,$$
where we have used the $K$ operator defined in \eqref{Koperator}.  Developing the square yields
\begin{eqnarray}
 \mathcal{G}_N(f,f)&=&
\frac{1}{N}\sum_{k=1}^N \int_{S^{N-1}(\sqrt{N})} m(v_k)[ \sum_{j\neq k} \varphi(v_j) - (N-1)K\varphi(v_k)]^2\dd \sigma
\nonumber\\
&=&  (N-1)  \int_{S^{N-1}(\sqrt{N})} m(v_1)\varphi^2(v_2)\dd \sigma\nonumber\\
&+& (N-1)(N-2)  \int_{S^{N-1}(\sqrt{N})} m(v_1)\varphi(v_2)\varphi(v_3)\dd \sigma\nonumber\\
&-& 2(N-1)^2  \int_{S^{N-1}(\sqrt{N})} m(v_1)K\varphi(v_1)\varphi(v_2)\dd \sigma\nonumber\\
&+& (N-1)^2  \int_{S^{N-1}(\sqrt{N})} m(v_1)(K\varphi(v_2))^2 \dd \sigma\nonumber\\ \ .
\end{eqnarray}

Since
$$ \int_{S^{N-1}(\sqrt{N})} m(v_1)K\varphi(v_1)\varphi(v_2)\dd \sigma = 
 \int_{S^{N-1}(\sqrt{N})} m(v_1)K\varphi(v_1)K\varphi(v_1)\dd \sigma\ ,$$
 and because
 $$\int_{S^{N-1}(\sqrt{N})} m(v_1)\varphi^2(v_2)\dd \sigma =
 \int_{S^{N-1}(\sqrt{N})} K m(v_1)\varphi^2(v_1)\dd \sigma\ ,$$
 \begin{eqnarray}\label{atl2}
\mathcal{G}_N(f,f)&=& 
 (N-1)  \int_{S^{N-1}(\sqrt{N})} Km(v_1)\varphi^2(v_1)\dd \sigma\nonumber\\
&+&  (N-1)(N-2)  \int_{S^{N-1}(\sqrt{N})} m(v_1)\varphi(v_2)\varphi(v_3)\dd \sigma\nonumber\\
&-& (N-1)^2  \int_{S^{N-1}(\sqrt{N})} m(v_1)(K\varphi(v_1))^2 \dd \sigma\nonumber\\
&=& I_1+I_2 + I_3\ .
\end{eqnarray}

Of the three integrals, $I_3$ is the easiest to estimate.  Noting that $m(v) \leq 1 +\gamma/(N-1)$, 
\begin{eqnarray}
I_3 &\geq& -(N-1)^2(1 + \frac{\gamma}{N-1}) \Vert K\varphi \Vert^2 \nonumber \\
&\geq& -(1 + \frac{\gamma}{N-1}) \frac{9}{(N+1)^2}  \Vert \varphi \Vert^2 \nonumber\ ,
\end{eqnarray}
and using Lemma \ref{gprop2}
we find the lower bound
\begin{equation}\label{I3b}
I_3 \geq -\left(1+\frac{\gamma}{N-1}\right)\left(1 - \frac{15}{(N+1)(N+3) }\right)^{-1}\frac{9}{(N+1)^2 N}\|f\|_2^2 \ .
\end{equation}

Next, we estimate $I_1$. Since, 
$$m(v) = \left[1+\frac{\gamma}{N-1}-\frac{1-\gamma}{(N-1)^2}\right] -
\left[ \frac{\gamma}{N-1} - 2\frac{1-\gamma}{(N-1)^2}\right]v^2 -
\frac{1-\gamma}{(N-1)^2}v^4\ ,$$
and because
$$K v^2 = -\frac{1}{N-1}(v^2 - N)\qquad{\rm and}\qquad 
K v^4 = \frac{3}{N^2-1}(v^2 - N)^2\ ,$$
\begin{multline}Km(v) = \left[1+\frac{\gamma}{N-1}-\frac{1-\gamma}{(N-1)^2}\right]\\ +
\left[ \frac{\gamma}{N-1} - 2\frac{1-\gamma}{(N-1)^2}\right]\frac{1}{N-1}(v^2 - N) -
\frac{1-\gamma}{(N-1)^2}\frac{3}{N^2-1}(v^2 - N)^2\ .
\end{multline}
Introducing $x = (v^2- N)/(N-1)$, so that $-N/(N-1) \leq x \leq 0$, we have
$$Km(v)  =   \left[1+\frac{\gamma}{N-1}-\frac{1-\gamma}{(N-1)^2}\right] + 
\left[ \frac{\gamma}{N-1} - 2\frac{1-\gamma}{(N-1)^2}\right]x - \frac{3(1-\gamma)}{N^2-1} x^2\ .$$
The right hand side is a concave function of $x$, so the minimum occurs at either $x= 0$
or
$x = -N/(N-1)$.  Direct computation shows that the minimum occurs at $x = -N/(N-1)$, and making a few 
simplifying estimates, we obtain the bound

\begin{equation}\label{sim}
Km(v) \geq  1 - \frac{2-\gamma}{(N-1)^2}\ .
\end{equation}

In fact, not making the simplifying assumptions, we have the stronger bound

\begin{equation}\label{detail}
Km(v) \geq  1 - \frac{2-\gamma}{(N-1)^2} + \frac{(1-\gamma)(2N-1)}{(N-1)^3(N+1)}\ .
\end{equation}

Thus, using \eqref{sim},
\begin{eqnarray}\label{I1b} I_1 &\geq& 
(N-1)\left( 1 - \frac{2-\gamma}{(N-1)^2}\right) \int_{S^{N-1}(\sqrt{N})} \varphi^2(v_1)\dd \sigma\nonumber\\
&\geq& \frac{N-1}{N}\left( 1 - \frac{2-\gamma}{(N-1)^2}\right)\left[ \|f\|_2^2 - N(N-1) \int_{S^{N-1}(\sqrt{N})} \varphi(v_2)
 \varphi(v_3)\dd \sigma\right]\ .
\end{eqnarray}
Adding the right side of (\ref{I1b}) to  $I_2$,  we obtain

\begin{eqnarray}  
I_1+I_2 
&\geq& \frac{N-1}{N}\left( 1 - \frac{2-\gamma}{(N-1)^2}\right)\|f\|_2^2  \nonumber\\
&+&(N-1) \int_{S^{N-1}(\sqrt{N})}\left[(N-2)m(v_1) - \left( 1 - \frac{2-\gamma}{(N-1)^2}\right)(N-1) \right] \varphi(v_2)
 \varphi(v_3)\dd \sigma\nonumber\\
 &=:&  \frac{N-1}{N}\left( 1 - \frac{2-\gamma}{(N-1)^2}\right)\|f\|_2^2 + J\label{Jdef}
\end{eqnarray}
where the last line defines $J$.
We compute
\begin{eqnarray}\left[(N-2)m(v_1) - \left( 1 - \frac{2-\gamma}{(N-1)^2}\right)(N-1) \right] &=&
-(1-\gamma)\left(1- \frac{N}{(N-1)^2}\right)\nonumber\\
&-& (N-2)\left[ \frac{\gamma}{N-1} - 2\frac{1-\gamma}{(N-1)^2}\right]v_1^2\nonumber\\
&-&
(N-2)\frac{1-\gamma}{(N-1)^2}v_1^4
\nonumber\\
\end{eqnarray}

Therefore,
\begin{eqnarray}
J
&=&   -(1-\gamma)\left(1- \frac{N}{(N-1)^2}\right) (N-1)(\varphi,K\varphi)\nonumber\\
&-&(N-2)\left[ \gamma- 2\frac{1-\gamma}{(N-1)}\right] \int_{S^{N-1}(\sqrt{N})}v_1^2
\varphi(v_2)\varphi(v_3)\dd \sigma \nonumber\\
& -&
(N-2)\frac{1-\gamma}{(N-1)} \int_{S^{N-1}(\sqrt{N})} v_1^4\varphi(v_2)\varphi(v_3)\dd \sigma \nonumber\ .
\end{eqnarray}

To estimate the moments we use the following lemma:
 \begin{lm} \label{moments}
\begin{eqnarray}
& &\int_{S^{N-1}(\sqrt{N})} v_1^2 \varphi(v_2) \varphi(v_3)   \dd \sigma =  \frac{([N - 2v^2]\varphi, K \varphi)}{N-2} \nonumber \\ 
& & \int_{S^{N-1}(\sqrt{N})} v_1^4 \varphi(v_2) \varphi(v_3)   \dd \sigma = \frac{([N^2-4Nv^2 +2v^4] \varphi, K \varphi) + 2(v^2 \varphi, K(v^2 \varphi)) }{(N-2)^2} 
\end{eqnarray}
where $(\psi, \varphi)$ denotes the inner product
\begin{equation}
  \int_{-\sqrt N}^{\sqrt N}\psi(v)\varphi(v){\rm d}\nu_N = \int_{S^{N-1}(\sqrt{N})} \psi(v_1) \varphi(v_1)   \dd \sigma  \ ,
 \end{equation}
 and  $\Vert \varphi \Vert^2 = (\varphi, \varphi)$.
\end{lm}

\begin{proof} Let $\psi_j(v) = v_1^{2j} $.  Letting $P_{\{2,3\}}$ denote the orthogonal projection in $\mathcal{H}$ 
onto the subspace of functions depending only on $v_2$ and $v_3$. Then 
$$\int_{S^{N-1}(\sqrt{N})} v_1^{2j} \varphi(v_2) \varphi(v_3)   \dd \sigma =    \int_{S^{N-1}(\sqrt{N})} P_{\{2,3\}}\psi_{j}(v_2,v_3) \varphi(v_2) \varphi(v_3)   \dd \sigma\ .$$
It is easy to compute $P_{\{2,3\}}\psi_j$ using formulas for the operator $K_{N,2}$ deduced in the final section. In Lemma~\ref{speccomp}, it is shown that
\begin{equation}\label{change10A}
P_{\{2,3\}}\psi_2(v_2,v_3) =   \frac{1}{N-3}[N- (v_2^2+v_3^2)]
\end{equation}
and
\begin{equation}\label{change11A}
P_{\{2,3\}}\psi_4(v_2,v_3) =  \frac{3}{(N-3)(N-1)}[N- (v_2^2+v_3^2)]^2\ .
\end{equation}
From here, the proof is a simple calculation using the definition of the $K$ operator. \end{proof}

Returning to the proof of Lemma~\ref{mainL}, we have
\begin{eqnarray}
J &\geq&   -(1-\gamma)\left(1- \frac{N}{(N-1)^2}\right) (N-1)(\varphi,K\varphi)\nonumber\\
&-&\left[ \gamma - 2\frac{1-\gamma}{(N-1)}\right]([N - 2v^2]\varphi, K \varphi) \nonumber\\
&-&
(1-\gamma)\frac{([N^2-4Nv^2 +2v^4] \varphi, K \varphi) + 2(v^2 \varphi, K(v^2 \varphi)) }{(N-1)(N-2)}  \nonumber\\
&=:& B_1 (\varphi,K\varphi)  + B_2(v^2\varphi,K\varphi) + B_3 [(v^4\varphi,K\varphi) +  (v^2\varphi,Kv^2 \varphi) ]\ . \nonumber
\end{eqnarray}

Collecting terms, we find
$$B_1 =  -(N-1)\left[   1 + \frac{(2+\gamma)N^2-(5\gamma+2)N + 2}{(N-1)^2(N-2) }   \right]\ .$$
Simple computations show that the quantity in square brackets is positive for all $0 \le \gamma \le 1$ and $N\geq 3$.

Likewise,
$$
B_2 = -\left[ 2\gamma + \frac{8(1-\gamma)}{(N-1)(N-2)}\right] 
\qquad{\rm and}\qquad 
B_3 = - \frac{2(1-\gamma)}{(N-1)(N-2)}\ .
$$
By Schwarz's inequality,
${\displaystyle 
|(v^4\varphi,K\varphi) | \le \Vert v^4 \varphi \Vert \Vert K\varphi \Vert \le \frac{3N^2}{N^2-1} \Vert \varphi \Vert^2
}$,
and likewise,
$$
 (v^2\varphi,Kv^2 \varphi)  \leq \frac{3N^2}{N^2-1} \Vert \varphi \Vert^2 \ .
 $$
In both cases we have used the fact that $\varphi$ as well as $v^2 \varphi$ are orthogonal to the constant function.
In a similar fashion we find that
${\displaystyle
|(v^2 \varphi, K\varphi)| \le \frac{3N}{N^2-1} \Vert \varphi \Vert^2
}$.
Collecting terms, we obtain
\begin{eqnarray}
J &\geq& \frac{3}{N^2-1}\left[ B_1  + NB_2 + 2N^2B_3\right ]\|\varphi\|^2\nonumber\\
&\geq& \frac{3}{N(N^2-1)}\left[ B_1  + NB_2 + 2N^2B_3\right]\left(1 - \frac{15}{(N+1)(N+3) }\right)^{-1} \|f\|_2^2\ .\nonumber
\end{eqnarray}
combining this estimate with (\ref{Jdef}) and (\ref{I3b}), and simplifying the sums, we obtain the desired bound on $I_1+I_2+I_3 =
\mathcal{G}_N(f,f)$. 
\end{proof}

\section{Lower bound for $\Delta_N$}

We now show how to decompose an admissible trial function $f$ in the variational formula for $\Delta_N$
into to components $g$ and $h$ where $g\in \mathcal{A}_N$, and $h$ satisfies $\langle h, P^{(\gamma)}h\rangle = 0$,
which means that $h$ makes no contribution to the negative term $\langle f, P^{(\gamma)}f\rangle$ in the induction bound
from Theorem~\ref{thm1}.  We use this, and further correlation estimates, to extend our lower bound for $\widehat{\Delta}_N$
into one for  ${\Delta}_N$.

\subsection{The trial function decomposition}

Let $\Pi$ denote the projection onto the space of functions orthogonal to the constants on 
$L^2(S^{N-1}(\sqrt{N}))$. Then the operator $\Pi P^{(\gamma)}\Pi$ is clearly self adjoint.

For any $f$ orthogonal to the constants, 
\begin{equation}\label{redn1}
\langle f, P^{(\gamma)} f \rangle =  \langle f,\Pi P^{(\gamma)} \Pi f \rangle\ .
\end{equation}
now decompose $f$ as $f = g+h$ where $h$ is in the null space of  $\Pi P^{(\gamma)}\Pi$,
and $g$ is in the range. Notice that $f$ and $g$ are orthogonal, so that 
$$\|f\|_2^2 = \|g\|_2^2 +    \|h\|_2^2\ .$$

By the definition of $h$ and (\ref{redn1}), 
\begin{equation}\label{redn2}
\langle f, P^{(\gamma)} f \rangle  = \langle g, P^{(\gamma)} g \rangle\ ,
\end{equation}
and hence
\begin{equation}\label{red2A}
  \int_{S^{N-1}(\sqrt{N})} W^{(\gamma)} f^2\dd \sigma
  - \langle f, P^{(\gamma)} f\rangle_{L^2(S^{N-1}(\sqrt{N}))} =
  \int_{S^{N-1}(\sqrt{N})} W^{(\gamma)} f^2\dd \sigma
  - \langle g, P^{(\gamma)} g\rangle_{L^2(S^{N-1}(\sqrt{N}))} \ .
  \end{equation}
  
 Notice that $h$ makes no contribution to the negative term on the right side of (\ref{red2A}). 
 In fact, an even stronger form of (\ref{redn2}) is true, and will be useful to us:

  \begin{lm}\label{ns1} Let $h$ be any function in $L^2((S^{N-1}(\sqrt{N}))$ that is orthogonal to the constants, and is in the null space of  $\Pi P^{(\gamma)}\Pi$. Then for each $k$,
  \begin{equation}\label{ns2}
  P_k h = 0\ .
  \end{equation}
  \end{lm}
  
\begin{proof}Since $\Pi h = h$, we have
  $$0 = \langle h, \Pi P^{(\gamma)}\Pi h\rangle =  \langle h,  P^{(\gamma)} h\rangle
  = \frac{1}{N}\sum_{k=1}^N\int_{S^{N-1}(\sqrt{N})}\left(\frac{N-v_k^2}{N-1}\right)^\gamma  |P_k h|^2\dd \sigma\ .$$
Since ${\displaystyle \left(\frac{N-v_k^2}{N-1}\right)^\gamma \ge 0}$ almost everywhere,
it must be the case that $|P_k h|^2$ vanishes identically. 
\end{proof}

 The other key feature of the decomposition is that $g\in \mathcal{A}_N$; i.e., here is a function $\varphi$ of
a single variable such that  $\varphi(v_k)\in L^2(S^{N-1}(\sqrt{N}))$ for each $k$ (or equivalently, for any $k$), and 
\begin{equation}\label{struc}
g( v) = \sum_{j=1}^N\varphi(v_j)\ .
\end{equation}
That is, the range of  $P^{(\gamma)}$ lies in the subspace $\mathcal{A}_N$ of $\mathcal{H}_N$ that figures in the definition (\ref{gapdef2}).  
This is because as long as $f$ is symmetric, so is $\Pi f$,  and  then $P^{(\gamma)}\Pi f$ has this form, and applying $\Pi$
preserves this form. Here we are using symmetry to ensure that we need just one and the same
 function $\varphi$ for each coordinate.

We now use the decomposition introduced at the beginning of this section to reduce the estimation of 
 $\Delta_N$ to the estimation of $\widehat{\Delta}_N$. 

We return to
\begin{equation}\label{ind2a}
\mathcal{F}_N(f,f) =  \frac{1}{N} \sum_{k=1}^N\left[   \int_{S^{N-1}(\sqrt{N})} w^{(\gamma)}(v_k)[f - P_kf]^2\dd \sigma
 \right]\ 
\end{equation}
and introduce the decomposition $f = g+h$. Then since $P_kh = 0$ for each $k$, 
$$ [f - P_kf]^2 = [g - P_k g]^2 + h^2 + 2[g - P_kg]h\ ,$$
\begin{equation}\label{main2}
\mathcal{F}_N(f,f) = \mathcal{F}_N(g,g)    + 
\frac{2}{N} \sum_{k=1}^N\left[   \int_{S^{N-1}(\sqrt{N})} w^{(\gamma)}(v_k)g h\dd \sigma \right]  +  \int_{S^{N-1}(\sqrt{N})}W^{(\gamma)}( v) h^2(v)\dd \sigma\ .
\end{equation}

We first estimate the cross terms. 

 \begin{lm}\label{crt} With $g$ and $h$ as above
 \begin{equation}\label{crt1}
\left| \frac{2}{N} \sum_{k=1}^N\left[   \int_{S^{N-1}(\sqrt{N})} w^{(\gamma)}(v_k)g h\dd \sigma \right]\right|
 \leq \frac{N-1}{N}\frac{C_N}{N^2}
 2\|g\|_2\|h\|_2
 \end{equation}
 where $C_N$ is given by (\ref{crt3}).
 \end{lm}

\begin{proof}
We rewrite \eqref{crt1} as
$$ 
\frac{2}{N} \sum_{k=1}^N\left[   \int_{S^{N-1}(\sqrt{N})} \left[w^{(\gamma)}(v_k)  - 1 - \gamma\frac{1-v_k^2}{N-1}\right]g h\dd \sigma \right]
$$
since 
${\displaystyle 
 \sum_{k=1}^N  \left[ 1 - \gamma\frac{1-v_k^2}{N-1}\right] = N}$, and $g$ and $h$ are orthogonal.  Introducing $g = \sum_{j=1}^N\varphi(v_j)$,
 the quantity becomes
 \begin{equation}\label{step}
 \frac{2}{N} \sum_{k=1}^N\sum_{j\neq k}\left[   \int_{S^{N-1}(\sqrt{N})} \left[w^{(\gamma)}(v_k)  - 1 - \gamma\frac{1-v_k^2}{N-1}\right]\varphi(v_j) h\dd \sigma \right]  
 \end{equation}
 where the term $j= k$ vanishes since $P_k h = 0$.  
 
 Next let  $P_{\{j,k\}}$ denote the orthogonal projection onto the subspace of functions depending only on $v_j$ and $v_k$. Evidently, we may replace $h$
  by $P_{\{j,k\}}h$ in each summand above. Then since
  $$\left[w^{(\gamma)}(v_k)  - 1 - \gamma\frac{1-v_k^2}{N-1}\right]^2 \leq \frac{(1-\gamma)^2}{(N-1)^4} (1-v_k^2)^4\ ,$$
the quantity in \eqref{step} is bounded in magnitude by   
 \begin{equation}\label{step2}
 (1-\gamma)\frac{1}{(N-1)^2}\frac{2}{N} \sum_{k=1}^N\sum_{j\neq k}\int_{S^{N-1}(\sqrt{N})} |1-v_k^2|^2|\varphi(v_j)| |P_{\{j,k\}} h|\dd \sigma 
 \end{equation}
 By the Schwarz inequality, and then the definition of the $K$ operator,  with $\psi(v)$ denoting the function $(1-v_k^2)^4$,
 $$\int_{S^{N-1}(\sqrt{N})} |1-v_k^2|^2|\varphi(v_j)| |P_{\{j,k\}} h|\dd \sigma  \leq 
 \left(\int_{S^{N-1}(\sqrt{N})} K\psi(v_j)|\varphi(v_j)|^2\right)^{1/2}\|P_{\{j,k\}} h\|_2\ .$$
 
 By the definition of the $K$ operator and $\psi$, $K\psi(v)$ is a convex function of $v^2$, and hence
 $$0 \leq K\psi(v) \leq\max\{ K\psi(0)\ ,\K\psi(\sqrt{N})\ .$$
 By direct computation
 \begin{multline}K\psi(v) = 1  - 4\frac{N-v^2}{N-1} + 18\frac{(N-v^2)^2}{(N-1)(N+1)}\\ 
 -60\frac{(N-v^2)^3}{(N-1)(N+1)(N+3)} + 105 \frac{(N-v^2)^4}{(N-1)(N+1)(N+3)(N+5)} \ .\nonumber\end{multline}
 Simple estimates show $K\psi(0) < K\psi(\sqrt{N}) \leq 60$
 for all $N$. Therefore, 
 $$ \left(\int_{S^{N-1}(\sqrt{N})} K\psi(v_j)|\varphi(v_j)|^2\right)^{1/2} \leq \sqrt{60}\|\varphi\|_2\|P_{\{j,k\}}h\|_2\ .$$
 Then, by symmetry, for each $j\neq k$, and  Theorem~\ref{correl2},
$$\|P_{\{j,k\}}h\|_2^2 = \langle h, P_{\{j,k\}}h\rangle = \ncht^{-1}\sum_{i< \ell} \langle h, P_{i,\ell}h\rangle \leq \left[\frac{2}{N-1} +\frac{8N}{(N-2)(N-4)^2}\right] \|h\|_2^2\ .$$ 
Further,  Lemma~\ref{gprop2} gives us
$$\|\varphi\|_2  \leq \left(1 - \frac{15}{ {(N+1)}(N+3)}\right)^{-1/2}\frac{\|g\|_2}{\sqrt{N}}\ .$$
Combining these yields the result. 
\end{proof}

  \begin{lm}\label{crtB} With  $h$ as above
 \begin{equation}\label{crt1B}
\int_{S^{N-1}(\sqrt{N})}W^{(\gamma)}( v) h^2(v)\dd \sigma  \geq \left(1 - \frac{1-\gamma}{N-1}\right)\|h\|_2^2\ .
 \end{equation}
 \end{lm}

\begin{proof} Simply use (\ref{squeeze}).\end{proof}

We come to:

\begin{proof}[Proof of Theorem~\ref{mainB}] From (\ref{main2}), Theorem~\ref{thm1}, Lemma~\ref{crt} and Lemma~\ref{crtB},
\begin{equation}\label{grow}
\frac{N}{N-1 }\mathcal{F}_N(f,f) \geq \left(1 - \frac{A_N}{N^2}\right)\|g\|_2^2 - \frac{C_N}{N^2}2\|g\|_2\|h\|_2 + \left(1+ \frac{\gamma N -1}{(N-1)^2}\right)\|h\|_2^2\ . 
\end{equation}
Then using $2\|g\|_2\|h\|_2 \leq \|f\|_2^2$,  we obtain
\begin{equation}\label{grow2}
\frac{N}{N-1 }\mathcal{F}_N(f,f) \geq \left(1 - \frac{A_N+ C_N}{N^2}\right)\|f\|_2^2\ ,
\end{equation}
from which the result follows. 
\end{proof} 

Since  $\lim_{N\to \infty}A_N =: A$ and $\lim_{N\to \infty}C_N =: C$ exist, for all $0\leq \gamma \leq 1$, there exists an $N_0$ so that  $N^2 > A_N + C_N$ for all $N \geq N_0$, and 
$$\prod_{j \geq N_0}^\infty \left(1 - \frac{A_N+ C_N}{N^2}\right) > 0\ .$$
By Theorem~\ref{uniform}
$$\Delta_{N_0} \geq 4 N^{\gamma -1} \left(\prod_{j=3}^{N_0}  \left[1 -   \frac{4j+1}{(j-1)^2(j+1)}\right]\right)>0 \ .$$
Altogether, we have
$$\liminf_{N\to\infty}\Delta_N \geq  4 N^{\gamma -1} \left(\prod_{j=3}^{N_0} 
 \left[1 -   \frac{4j+1}{(j-1)^2(j+1)}\right]\right) \prod_{j \geq N_0}^\infty \left(1 - \frac{A_N+ C_N}{N^2}\right) > 0\ ,$$
 which proves the Kac conjecture for $0 < \gamma < 1$. 
 
 \subsection{The structure of the gap eigenfunction}
 
  The exact computation of $\Delta_N$ for $\gamma =0$ \cite{CCL00} shows that in this case, $\Delta_N < \Delta_{N-1}$ for all $N$. 
 It seems quite plausible that for {\em all}  $\gamma$, $\Delta_N$ is monotone decreasing in $N$, but we have not been able to show this. 
 All of our work so far in this paper has focused on lower bounds, for the obvious reasons. 
 
 Nonetheless,  the conjectured monotonicity of $\Delta_N$ would have a significant consequence. Fix $\gamma > 0$. For any $N\geq 3$, let $f_N$
 be a normalized  eigenfunction of $L_N$ with $L_N f_N = -\Delta_N f_N$.  Let
 $f_N = g_N + h_N$ be the trial function decomposition of $f_N$.  
 Define $\alpha_N$ by
 $$\alpha_N := \|h_N\|^2\ .$$
In passing from (\ref{grow}) to (\ref{grow2}) we have simply discarded the term ${\displaystyle  \frac{\gamma N -1}{(N-1)^2}\|h\|^2}$. Taking $f = f_N$ and keeping this term, we deduce
$$\Delta_N \geq   \left(1 - \frac{A_N+ C_N}{N^2}  + \alpha_N \frac{\gamma N -1}{(N-1)^2}\right)\Delta_{N-1}\ .$$
If we then knew that $\Delta_N \leq \Delta_{N-1}$ for all $N$, we could conclude that
$$
\alpha_N \leq \frac{A_N+C_N}{\gamma N}$$
for all $N$, and since $\lim_{N\to\infty}A_N = A$, and   $\lim_{N\to\infty}C_N = C$  exist and are finite, this would mean that $\alpha_N = \mathcal{O}(1/N)$. 
Instead, we can prove:

\begin{prop} There is an infinite sequence of integers $\{N_k\}$ such that for each $k$, $\alpha_{N_k} \leq 1/\log(N_k)$. 
\end{prop}

\begin{proof} If this were not the case, then 
$$\sum_{N=3}^\infty \left[ - \frac{A_N+ C_N}{N^2}  + \alpha_N \frac{\gamma N -1}{(N-1)^2}\right] $$
would diverge to $+\infty$, and this would imply that $\Delta_N$ would increase without bound. However, the trial function calculations
in the next section show that this is not the case. 
\end{proof} 

The next bound shows that if one could strengthen this to $\alpha_N = o(1/N^\gamma)$ for all large $N$, one would conclude that 
$$\lim_{N\to\infty}(\widehat \Delta_N - \Delta_N) = 0\ .$$

\begin{prop} Suppose that $\alpha_N = o(1/N^\gamma)$. 
There is a computable constant $C$ so that  for all sufficiently large $N$,
$$ \Delta_N \geq   \widehat{\Delta}_N(1-\alpha_N) - C\sqrt{N^\gamma\alpha_N}\ .$$
\end{prop}

\begin{proof}  Let $f_N$  be a normalized gap eigenfunction for $L_N$, and let $f_N = g_N + h_N$ be its decomposition as above. 
Then
$$\Delta_N = \mathcal{E}_N(f_N,f_N) =   \mathcal{E}_N(g_N,g_N)  +2  \mathcal{E}_N(g_N,h_N) + \mathcal{E}_N(h_N,h_N)\ .$$
We note that the cross term cannot be positive since otherwise replacing $f_N$ by $g_N - h_N$ would  lower the value 
of   $\mathcal{E}_N(f_N,f_N)$ while respecting the constraints. We shall discard the term $ \mathcal{E}_N(h_N,h_N)$, and estimate the cross
term from below. To do this, write $g(v) = 
\sum_{j=1}^N\varphi(v_j)$ as above, and define
$$
g_{i,j} = (v_i^2+v_j^2)^\gamma\frac{1}{2\pi}\int_{-\pi}^\pi \left[\varphi(v_i\cos\theta + v_j\sin\theta) +  \varphi(-v_i\sin\theta - v_j\cos\theta) - \varphi(v_i) - \varphi(v_j)\right]{\rm d}\theta\ ,$$
so that
$$L_N g_N = N \ncht^{-1} \sum_{i<j} g_{i,j}\ .$$

We have
\begin{eqnarray}  -\mathcal{E}_N(g_N,h_N)  &=& \langle L_N g_N, h_N\rangle\nonumber\\
&=&    N \ncht^{-1} \sum_{i<j} \langle g_{i,j} , h_N\rangle\nonumber\\
&=&    N \ncht^{-1} \sum_{i<j} \langle g_{i,j} , P_{\{i,j\}}h_N\rangle\nonumber\\
&\leq& N\left( \ncht^{-1} \sum_{i<j} \|g_{i,j}\|^2  \right) ^{1/2}  \left( \ncht^{-1} \sum_{i<j} \|P_{\{i,j\}}h_N\|^2  \right) ^{1/2} \nonumber
\end{eqnarray}
By Theorem~\ref{correl2}, 
$$ \ncht^{-1} \sum_{i<j} \|P_{\{i,j\}}\|^2  \leq \frac{C}{N} \|h_N\|^2 =  \frac{C}{N}\alpha_N\ .$$
Also, using the bound $(v_i^2+v_j^2)^\gamma \leq N^\gamma$, 
$$N  \ncht^{-1} \sum_{i<j} \|g_{i,j}\|^2  \leq  N^\gamma\mathcal{E}_N(g_N,g_N)\ .$$
Altogether, we have
$$\Delta_N \geq   \mathcal{E}_N(g_N,g_N) - 2\sqrt{ CN^\gamma \alpha_N}\sqrt{\mathcal{E}_N(g_N,g_N)}\ .$$
Now, $\mathcal{E}_N(g_N,g_N) \geq \widehat\Delta_N\|g_N\|^2 =  \widehat\Delta_N(1-\alpha_N)$. Hence for all sufficiently large
$N$, the difference above decreases if we replace  $\mathcal{E}_N(g_N,g_N)$ by $\widehat\Delta_N(1-\alpha_N)$.
Renaming the constant, this yields the desired bound.
\end{proof}

\section{Spectral gap for the linearized collision operator}

Our main goal in this section is to prove Theorem~\ref{apl}, which says that for  $0 \leq \gamma < 1$,   
$$\Lambda \geq \limsup_{N\to\infty}\widehat \Delta_N\ ,$$
where $\Lambda$ is the spectral gap for the linearized Kac-Boltzmann equation.  

In the proof, we shall use 
several lemmas. Let $\rho_N(v)$ be the probability density on $\R$ defined by
$$\rho_N(v)\dd v = \dd \nu_{M,1}\ .$$
That is,  $\rho_N$ is the density of the one dimensional marginal of the uniform probability measure on $S^{N-1}(\sqrt{N})$.  Writing (\ref{change2}) out explicitly, one has
\begin{equation}\label{lin5}
\rho_N(v)\dd v = K_N\left(1 - \frac{v^2}{N}\right)_+^{(N-3)/2}\qquad{\rm where}\qquad 
K_N = \frac{1}{\sqrt{N\pi}}\frac{\Gamma(N/2)}{\Gamma((N-1)/2)} \ .
\end{equation}
By Stirlings formula,
\begin{equation}\label{lin6}
K_N = \frac{1}{\sqrt{2\pi}}\left(1 + {\mathcal O}\left(\frac{1}{N}\right)\right)\ .
\end{equation}

\begin{lm}\label{llin1} There is a constant $C$ independent of $N$ such that
\begin{equation}\label{lin7}
\rho_N(v) \leq CM(v)\quad{\rm for\ all}\quad v\in \R\ ,
\end{equation}
where $M(v)$ is given by \eqref{max2}.
\end{lm}

\begin{proof}
Using the elementary estimate ${\displaystyle e^{-v^2} \geq \left(1 - \frac{v^2}{N}\right)_+^N}$,
valid for all $v$, 
we deduce
\begin{equation}\label{lin8}
\rho_N(v) \leq K_N e^{-v^2/2} e^{3v^2/2N} = K_N \sqrt{2\pi} e^{3v^2/2N} M(v) \ .
\end{equation}
Then since $v^2\leq N$ for all $v$ in  the support of $\rho_N$ , we have
$$\rho_N(v) \leq K_N \sqrt{2\pi} e^{3/2} M(v)\ .$$
The claim now follows from (\ref{lin6}). In fact, we see that for all sufficiently large $N$,
the constant $C = e^2$ would suffice. 
\end{proof}

In particular, it follows from Lemma~\ref{llin1} that $\rho_N(v)/M(v)$ is uniformly bounded on $\R$.
We now show that this ratio converges to $1$ uniformly on bounded intervals as $N$ tends to infinity.

\begin{lm}\label{llin2} For all $R>0$, 
\begin{equation}\label{lin9}
\lim_{N\to\infty}\left(\sup_{-R\leq v\leq R}\left|\frac{\rho_N(v)}{M(v)} -1\right|\right) = 0\ .
\end{equation}
\end{lm}

\begin{proof} From (\ref{lin8}), one has the upper bound
$$\frac{\rho_N(v)}{M(v)} -1 \leq   K_N \sqrt{2\pi} e^{3R^2/2N}  - 1$$
for all $v \in [-R,R]$. By (\ref{lin6}), the right hand side is ${\mathcal O}(1/N)$.

For the lower bound, we start from the elementary estimate
\begin{equation}\label{lin10}
 \left(1 - \frac{v^2}{N}\right)^N \geq  e^{-v^2 - v^4/N}\qquad{\rm whenever}\qquad
v^2 \leq \frac{N}{2}\ ,
\end{equation}
Thus,
$$\frac{\rho_N(v)}{M(v)} -1 \geq  K_N \sqrt{2\pi} e^{-R^4/2N}  - 1\ .$$
Once again, by  (\ref{lin6}), the right hand side is ${\mathcal O}(1/N)$. 
\end{proof}

Notice the proof of Lemma~\ref{llin2} gives a little more than is stated: One could let
$R$ grow with $N$ like $N^\alpha$ for any $\alpha < 1/4$.

To apply these lemmas, let $\varphi$ be any unit vector in $L^2(\R,M(v)\dd v)$ that is orthogonal to both
$1$ and $v^2$. Define $a_N$, $b_N$ and $\varphi_N$ by
$$a_N = \int_{\R}\varphi(v)\rho_N(v)\dd v\qquad b_N = \left(\int_{\R}(v^2-1)^2\rho_N(v)\dd v\right)^{-1/2}
\int_{\R}\varphi(v)(v^2-1)\rho_N(v)\dd v\ ,$$
and
$$\varphi_N(v) =  \varphi(v) - a_N - b_N\left(\int_{\R}(v^2-1)^2\rho_N(v)\dd v\right)^{-1/2}
(1-v^2)\ .$$
Notice that since $1$ and $v_1^2 -1$ are orthogonal with respect to the uniform probability measure
on $S^{N-1}(\sqrt{N})$, $1$ and $v^2-1$ are orthogonal with respect to $\dd \nu_{N,1}(v) = \rho_N(v)\dd v$.  Thus, for each $j$, $\varphi_N(v_j)$ is ortogonal to both $1$ and $v_j^2$ with
respect to the uniform probability measure
on $S^{N-1}(\sqrt{N})$.  

The next lemma concerns the trial function that we shall use to prove Theorem~\ref{apl}

\begin{lm}\label{llin3} Define the function $f_N( v)$ by
\begin{equation}\label{lin18}
f_N( v) = \frac{1}{\sqrt{N}}\sum_{j=1}^N  \varphi_N(v_j) \ .
\end{equation}
Then for each $N$, $f \in L^2(S^{N-1}(\sqrt{N})$ and is orthogonal to the constants. Moreover,
\begin{equation}\label{lin19}
\lim_{N\to\infty} \int_{S^{N-1}(\sqrt{N})}f_N^2( v)\dd \sigma = 1\ .
\end{equation}
\end{lm}

\begin{proof} By Lemma~\ref{llin1},  there is a finite constant so that $\rho_N(v) \leq CM(v)$,
and thus
$$\int_{S^{N-1}(\sqrt{N})}
 \varphi^2(v_1)\dd \sigma = \int_\R \varphi^2(v)\rho_N(v)\dd v \leq C
  \int_\R \varphi^2(v)M(v)\dd v \leq C\ .$$
  This proves the square integrability, and then the orthogonality statement follows by construction. 
  
Next, by Lemma~\ref{gprop2}, 
\begin{equation}\label{lin20}
\left(1- \frac{15}{(N+1)(N+3)}\right)\int_{S^{N-1}(\sqrt{N})}\varphi_N^2(v_1)\dd \sigma \leq 
\int_{S^{N-1}(\sqrt{N})}f_N^2( v)\dd \sigma \leq
\left(1+ \frac{3}{N+1}\right)\int_{S^{N-1}(\sqrt{N})}\varphi_N^2(v_1)\dd \sigma\ .
\end{equation}

Also, by construction,
$$\int_{S^{N-1}(\sqrt{N})}\varphi_N^2(v_1)\dd \sigma  = 
\int_{\R}\varphi^2(v)\rho_N(v)\dd v - a_N^2 - b_N^2\ .$$
Again by Lemma~\ref{llin1},  $\rho_N(v) \leq CM(v)$ for all $v$, and then since
$\lim_{N\to\infty}\rho_N(v) = M(v)$ for all $v$, the Dominated Convergence Theorem yields
$$\lim_{N\to\infty}\int_{\R}\varphi^2(v)\rho_N(v)\dd v  = 
\int_{\R}\varphi^2(v)M(v)\dd v =1\ .$$
\end{proof}

\begin{proof}[Proof of Theorem~\ref{apl}]
By symmetry and direct computation,
$$\E(f_N,f_N) = N \int_{S^{N-1}(\sqrt{N})}(v_1^2+v_2^2)^\gamma  f_N( v)[ 
f_N( v) - f_N^{1,2}( v)]\dd \sigma\ ,$$
and also
$$f_N( v) - f_N^{1,2}( v) = \frac{1}{\sqrt{N}}\left(\varphi_N(v_1) +\varphi_N(v_2) - 
2\varphi_N^{(1,2)}(v_1,v_2)\right)$$
where
$$ \varphi_N^{(1,2)}(v_1,v_2)
\frac{1}{2\pi}\int_{-\pi}^\pi \varphi_N(\cos\theta v_1 + \sin\theta v_2)\dd\theta \ .$$

Thus,
\begin{eqnarray}
\E(f_N,f_N) &=& 
\int_{S^{N-1}(\sqrt{N})}(v_1^2+v_2^2)^\gamma  \left(\varphi_N(v_1) + \varphi_N(v_2)\right)
\left(\varphi_N(v_1) +\varphi_N(v_2) - 2\varphi_N^{(1,2)}(v_1,v_2))\right)
\dd \sigma\nonumber\\
&+&\sum_{j=3}^N \varphi_N(v_j)
\int_{S^{N-1}(\sqrt{N})}(v_1^2+v_2^2)^\gamma  
\left(\varphi_N(v_1) +\varphi_N(v_2) - 2\varphi_N^{(1,2)}(v_1,v_2)) \right)
\dd \sigma\ .
\end{eqnarray}

Using our results on the spectrum of $K_{(2)}$, the orthogonality properties of $\varphi_N$,
and the trivial bound  $(v_1^2+v_2^2)^\gamma \leq N^\gamma$, we have that
$$\left|\sum_{j=3}^N \varphi_N(v_j)
\int_{S^{N-1}(\sqrt{N})}(v_1^2+v_2^2)^\gamma  
\left(\varphi_N(v_1) +\varphi_N(v_2) - 2\varphi_N^{(1,2)}(v_1,v_2)) \right)
\dd \sigma\right| \leq \frac{C}{N^{1-\gamma}}\ .$$
Next, simple computations show that for each $N$, 
$$\left(\varphi_N(v_1) +\varphi_N(v_2) - 2\varphi_N^{(1,2)}(v_1,v_2))\right) = 
\left(\varphi(v_1) +\varphi(v_2) - 2\varphi^{(1,2)}(v_1,v_2))\right)\ ,$$
and then that the value of 
$$
\int_{S^{N-1}(\sqrt{N})}(v_1^2+v_2^2)^\gamma  \left(\varphi_N(v_1) + \varphi_N(v_2)\right)
\left(\varphi_N(v_1) +\varphi_N(v_2) - 2\varphi_N^{(1,2)}(v_1,v_2))\right)$$
is unchanged upon replacing $\varphi_N$ by $\varphi$. Finally, a simple dominated convergence argument based on the two dimension analog of $\rho_n(v) \leq CM(v)$ shows that
$$\lim_{N\to\infty} \E(f_N,f_N) = \langle \varphi, {\mathcal L}\varphi\rangle_{L^2(\R,M(v)\dd v)}\ .$$
Combining this with Lemma~\ref{llin3}, we have
$$
 \langle \varphi, {\mathcal L}\varphi\rangle_{L^2(\R,M(v)\dd v)} = 
 \lim_{N\to\infty}\frac{ \E(f_N,f_N)}{\|f_N\|_2} \geq   \limsup_{N\to\infty}\widehat\Delta_N\ .$$
 As $\varphi$ is an arbitrary trial function in the variational definition of $\Lambda$, this proves
 Theorem~\ref{apl}. 
\end{proof}

\section{Correlation operators on the sphere}

In this section we work on $S^{N-1} = S^{N-1}(1)$, the unit sphere in $\R^N$, and we let
$\dd \sigma_N$ denote the uniform probability on $S^{N-1}$.  
The relation between the uniform probability measure on  $S^{N-1}(\sqrt{N})$ that was studied in the last section
suggests that the coordinate functions on $(S^{N-1},\dd \sigma_N)$, regarded as a probability space, should be ``nearly independent'' for large $N$, at least if one does not ``look at too many of them at once''. 
Our goal in this section is to prove results that precisely express, in a quantified manner,  these assertions.
Using the unitary rescaling operation (\ref{scaleC}) 
we shall then be able to apply our results in $L^2(S^{N-1}(\sqrt{N}))$. However, the derivation of the results shall be easier if we work on the unit sphere. 

We begin by introducing some notation. Let $\h$ denote the Hilbert space  $L^2(S^{N-1},\dd \sigma_N)$.
Given any non-empty subset $A\subset\{1,\dots,N\}$, let $\h_A$ denote the subspace of $\h$
that is the closure of the set of all polynomials in the variables $v_j$ with $j\in A$. That is, loosely
speaking,  $\h_A$ is the subspace of functions only depending on coordinates with indices in $A$. 
Let $P_A$ denote the orthogonal projection of $\h$ onto $\h_A$. In the special case that $A = \{j\}$,
we simple write $P_j$ to denote this projection. This usage is consistent with our previous use of
the notation $P_j$. 

For each $m=1,\dots,N-1$, let $\pi_{\{1,\dots,m\}}$ denote the map
\begin{equation}\label{pimd}
\pi_{\{1,\dots,m\}}( v) = (v_1,\dots,v_m)\ .
\end{equation}
For simplicity, we write $\pi_j$ to mean $\pi_{\{j\}}$.
The image of  $S^{N-1}$ under $\pi_{\{1,\dots,m\}}$ is the unit ball $B^m$ in $\R^m$, and by a standard computation, for any continuous function $\psi$ on $B^m$,
\begin{equation}\label{change}
\int_{S^{N-1}}\psi[\pi_{\{1,\dots,m\}}( v) ]\dd \sigma = \int_{B^m}\psi[w]\dd \nu_{N,m}(w)\ ,
\end{equation}
where 
\begin{equation}\label{change2}
\dd \nu_{N,m}(w) = \frac{|S^{N- m-1}|}{|S^{N-1}|} \left(1 - |w|^2\right)^{(N - m -2)/2}\dd w
\end{equation}
where $|S^{k-1}|$ denote the surface area of $S^{k-1}$ in $\R^k$.
We define $\K_{N,m}$ to be the Hilbert space $L^2(B^m,  \nu_{N,m})$.

Now define $\phi_{N,m}: S^{N-m-1}\times B^m \to S^{N-1}$ by
\begin{equation}\label{change3}
\phi_{N,m}(y,w) = (\sqrt{1 - |w|^2}y_1,\dots,\sqrt{1 - |w|^2}y_{N-m},w_1,\dots,w_m)\ .
\end{equation}
Then for any continuous function $h$ on $S^{N-1}$,
\begin{equation}\label{change4}
\int_{S^{N-1}}h[x]\dd \sigma(x) = \int_{B^m}\left[ \int_{S^{N-m-1}}h[\phi_{N,m}(y,w)] \dd \sigma(y)\right]
\dd \nu_{N,m}(w)\ .
\end{equation}

For any $m\leq N/2$, we define the {\em $m$-particle correlation operator} $K_{(N,m)}$ 
as a self adjoint operator on $\K_{N,m}$ through
\begin{equation}\label{change5}
\langle K_{(N,m)} f,g\rangle_{\K_{N,m}} = 
\int_{S^{N-1}}f(v_1,\dots,v_m)g(v_{N-m+1}, \dots,v_N)\dd \sigma\ .
\end{equation}
Then using
(\ref{change3}), we find
\begin{equation}\label{change6}
K_{(N,m)}f(w_1,\dots,w_m) = \int_{S^{N-m-1}}f(\sqrt{1 - |w|^2}y_1,\dots,\sqrt{1 - |w|^2}y_{m})
\dd \nu_{N-m,m}(y)\ .
\end{equation}

\begin{thm}\label{kprops}
For all $m \leq N$ and all $\N\geq 3$,  the non zero eigenvalues of  $K_{(N,m)}$ are given by
\begin{equation}\label{eigenform1}
\kappa_{N,m}(k) = (-1)^k \frac{\Gamma\left(\frac{N-m}{2}\right)}
{\Gamma\left(\frac{N-m}{2} + k\right)}
\frac{\Gamma\left(\frac{m}{2}+k\right)}
{\Gamma\left(\frac{m}{2}\right)}
\qquad{for}\qquad k = 0,1,2,\dots\ .
\end{equation}
For each $k\geq 0$, one has
\begin{equation}\label{eigenform2}
|\kappa_{N,m}(k)| > |\kappa_{N,m}(k+1)|\ .
\end{equation}
Moreover, for each $k$, the  eigenspace corresponding to  $\kappa_{N,m}(k)$ is one dimensional, and is spanned by a polynomial of degree $k$ in $|y|^2$, $y\in B^m$.
In particular, the eigenspace  corresponding to  $\kappa_{N,m}(1)$ is spanned by
\begin{equation}\label{eigenform3}
|y|^2 - m/N\ .
\end{equation}
\end{thm}

\begin{proof}
Observe from (\ref{change6})
 that $K_{(N,m)}f$ is always radial.  Moreover, by the symmetries of $\dd \nu_{N-m,m}(y)$,
if $f$ is a polynomial of total degree $k$ in $y_1,\dots,y_m$, so is  $K_{(N,m)}f$.
Thus, the eigenfunctions of  $K_{(N,m)}$ that are not in the null space of $K_{(N,m)}$ are
polynomials in $|y|^2$, where $y\in B^m$ and hence $0 \le |y|^2 \le 1$. The polynomials are
thus identified as (shifted and scaled) Jacobi polynomials. But even more easily, it is easy to see
that the eigenvalue $\kappa_{N,m}(k)$ of  $K_{(N,m)}$ that corresponds
to the eigenfunction that is a polynomial of order $|v|^{2k}$ is given by 
\begin{equation}\label{eigenform}
\kappa_{N,m}(k) = (-1)^k \int_{B^m}|y|^{2k} \dd \nu_{N-m,m}(y)\ .
\end{equation}

This integral can easily be evaluated in terms of the Beta function, $B(x,y)$. Using (\ref{change2}), we have
\begin{eqnarray}
\int_{B^m}|y|^{2k} \dd \nu_{N-m,m}(y) 
&=& \frac{|S^{N-2m-1}||S^{m-1}|}{|S^{N-1}|}\int_0^1 (1 - r^2)^{(N - 2m -2)/2}(r^2)^{(m-1+2k)/2}\dd r\nonumber\\
&=& \frac{|S^{N-2m-1}||S^{m-1}|}{2|S^{N-1}|}\int_0^1 (1 - x)^{(N - 2m -2)/2}(x)^{(m+2k-2)/2}\dd x\nonumber\\
&=& \frac{|S^{N-2m-1}||S^{m-1}|}{2|S^{N-1}|}B\left(\frac{N}{2} -m, \frac{m}{2} + k\right)\ .
\end{eqnarray}
Since the right hand side equals one for $k=0$, it follows that
$$ \frac{|S^{N-2m-1}||S^{m-1}|}{2|S^{N-1}|} = \left[B\left(\frac{N}{2} -m, \frac{m}{2}\right)\right]^{-1}\ .$$
This leads to the explicit formula
$$|\kappa_{N,m}(k) | = \frac{B(N/2 - m, m/2+k)}{B(N/2 - m,m/2)}\ .$$
Finally, using the identity ${\displaystyle B(x,y) = \frac{\Gamma(x)\Gamma(y)}{\Gamma(x+y)}}$, 
one arrives at  (\ref{eigenform1}).

While (\ref{eigenform2}) can be deduced from (\ref{eigenform1}), it is much simpler to observe that since
$|y|^{2k} >  |y|^{2k+2}$ almost everywhere on $B^m$, (\ref{eigenform2}) follows directly
from (\ref{eigenform}).
In particular, there is no degeneracy in the non-zero spectrum,
and  each eigenvalue $\kappa_{N,m}(k)$ has a one dimensional  eigenspace  spanned by a polynomal of degree $k$ in $|y|^2$. 
Since functions in the eigenspace for $\kappa_{N,m}(1)$ must be orthogonal to the constants, it
follows that this eigenspace is spanned by  the function given in (\ref{eigenform3}). 
\end{proof}

Using (\ref{eigenform3}) and the identity $\Gamma(x+1) = x\Gamma(x)$, one readily computes
\begin{equation}\label{change50}
\kappa_{N,m}(1) = -\frac{m}{N-m}
\end{equation}
and
\begin{equation}\label{change51}
\kappa_{N,m}(2) = \frac{m(m+2)}{(N-m)(N-m+2)}\ .
\end{equation}

We now relate these results to the idea that the coordinate functions on $(S^{N-1},\dd \sigma_N)$, regarded as a probability space, should be ``nearly independent'' for large $N$, at least if one does not ``look at too many of them at once''. 

Fix any $m \leq N/2$, and let $A$ and $B$ be two {\em disjoint} subsets of $\{1,\dots,N\}$, each of cardinality $m$. 
Let $f$ and $g$ be any two functions in $L^2(B^m,\dd \nu_{N,m})$ that are orthogonal to the constants. 
Then, by the definitions of $\dd \nu_{N,m}$, $K_{(N,m)}$ and symmetry,
\begin{equation}\label{eigenform2A}
\int_{S^{N-1}}\left(f\circ \pi_A\right)\left( g\circ\pi_B\right)\dd \sigma_N = \int_{B^m}f(y)[K_{(N,m)}g](y)\dd \nu_{N,m}(y)\ .
\end{equation}
If the coordinate functions were independent on $(S^{N-1},\dd \sigma_N)$, then the left hand side of
(\ref{eigenform2}) would equal
$$
\int_{S^{N-1}}\left(f\circ \pi_A \right)\dd \sigma_N \int_{S^{N-1}} \left(g\circ\pi_B\right)\dd \sigma_N = 
 \int_{B^m}f(y)\dd \nu_{N,m}(y)\int_{B^m}g(y)\dd \nu_{N,m}(y)= 0\ .
$$
That is, the random variables $\left(f\circ \pi_A \right)$ and $\left(g\circ \pi_B \right)$ would be uncorrelated. 
However, the coordinate functions are not independent on $(S^{N-1},\dd \sigma_N)$, and the size of
the right hand side of
(\ref{eigenform2}) measures the resulting correlations between $\left(f\circ \pi_A \right)$ and $\left(g\circ \pi_B \right)$.
By Theorem~\ref{kprops} and the computation (\ref{change50}), we have
$$\left|\int_{B^m}f(y)[K_{(N,m)}g](y)\dd \nu_{N,m}(y)\right| \leq  \frac{m}{N-m}\|f\|_{L^2(B^m,\dd \nu_{N,m})}
\|g\|_{L^2(B^m,\dd \nu_{N,m})}\ .$$
That is,
\begin{equation}\label{change52}
\left|\int_{S^{N-1}}\left(f\circ \pi_A\right)\left( g\circ\pi_B\right)\dd \sigma_N\right| \leq 
 \frac{m}{N-m} \left(\int_{S^{N-1}}\left(f\circ \pi_A \right)^2\dd \sigma_N\right)^{1/2}
  \left(\int_{S^{N-1}}\left(g\circ \pi_B \right)^2\dd \sigma_N\right)^{1/2}\ .
\end{equation}

One sees from (\ref{change52}) that if $m$ is small compared to $N$, then 
$\left(f\circ \pi_A \right)$ and $\left(g\circ \pi_B \right)$ have only a small correlation.  However, if we look at too many variables at once, so that $m$ is not small compared with $N$, there can be significant correlation. For example,
if $N$ is even and $m = N/2$, then $m/(N - m) =1$, and  $\left(f\circ \pi_A \right)$ and $\left(g\circ \pi_B \right)$
can be {\em fully} correlated. 
Indeed, if $N = 2m$ and one knows the value of $v_1^2 + \cdots + v_m^2$, then knows the value of
$v_{m+1}^2+ \cdots + v_N^2$ with complete certainty. 

Let us now consider the cases $m=1$ and $m=2$ with $N$ large, so that correlations will be small. 
Let $f$ be any function in $\h$ that is orthogonal to the constants. Then for each $k=1,\dots,N$, $P_kf$ is also orthogonal to the constants, and for $j\neq k$, $P_jf$ and $P_kf$ would be nearly orthogonal.
If they were {\em exactly} orthogonal, Bessel's inequality would imply that
$$\sum_{k=1}^N \|P_kf\|^2_\h$$
would be no larger than  ${\displaystyle \|f\|^2_\h}$.
It turns out that  for $m= 1$ and large $N$,  the correlations are small
enough that  something almost as good as this is true:

\begin{thm}\label{correl1}
Let $f$ be any function in $\h$ that is orthogonal to the constants.  Then
\begin{equation*}
\frac{1}{N} \sum_{k=1}^N \|P_kf\|^2_\h \leq \left(\frac{1}{N} + \frac{3}{N(N+1)}\right)
\|f\|^2_\h\ .
\end{equation*}
\end{thm}

 Theorem~\ref{correl1} is proved in \cite{CCL03}. 
 The following theorem is a companion to it for
 $m=2$. For $f$ orthogonal to the constants and $N$ large, $P_{\{1,2\}}f$, 
 $P_{\{3,4\}}f$, $P_{\{5,6\}}f$, and so forth, should be nearly orthogonal and so one might expect that
 $$\sum_{\{j\ :\ 2 \leq 2j \leq N\}}\|P_{\{2j-1,2j\}}f \|^2_\h$$
 should not be much larger than $\|f\|^2_\h$. Since there are essentially 
 $N/2$ terms in the sum, and the measure $\dd \sigma_N$ is permutation invariant, one might 
 expect a result of the following type:
 
 \begin{thm}\label{correl2}  For all  $N \ge 3$ and all $f$  orthogonal to the constants,
\begin{equation*}
\ncht^{-1} \sum_{i<j} \|P_{\{i,j\}}f\|^2  \leq \left[\frac{2}{N-1} +\frac{8N}{(N-2)(N-4)^2}\right] \|f\|^2\ .
\end{equation*}

\end{thm}

\begin{proof}Let $M$ denote the integer part of $N/2$.
Let ${\mathcal A}$ denote any set of 
$M$ {\em disjoint} pairs of indices in $\{1,\dots,N\}$.  Let
$P_{\mathcal A}$ denote the self-adjoint operator
$$P_{\mathcal A} = \frac{1}{M}\sum_{\alpha\in {\mathcal A}}P_\alpha\ .$$

We wish to bound 
\begin{equation}\label{var1}
\sup\left \{  \langle f, P_{\mathcal A} f\rangle \ :\  \|f\|_2 =1\ ,\ \langle f,1\rangle = 0 \quad{\rm and}\quad  f\ {\rm is\ symmetric}\ \right\}
\end{equation}
Define the function
$$\eta(v_1,v_2) = \frac{2}{N} - v_1^2-v_2^2$$
on the unit ball in $\R^2$.  This function spans the eigenspace of $K_{(N,2)}$ with eigenvalue $\kappa_{N,2} = -2/(N-2)$. 
For the pair $\gamma = (\gamma_1,\gamma_2)$, define
$$\xi_\gamma(v) =  \eta(v_{\gamma_1},v_{\gamma_2}) \ ,$$
Then since the average of the $\xi_\gamma$ over all pairs is zero, for $f$ is symmetric, $f$ is orthogonal to each $\xi_\gamma$. 
In particular, such an $f$ is orthogonal to each $\xi_\alpha$ with $\alpha\in \mathcal{A}$.  Since the pairs in $\mathcal{A}$
are disjoint, and since $\eta$ is an eigenfunction of $K_{(N,2)}$, it is evident that the span of $\{\xi_\alpha\ :\ \alpha \in \mathcal{A}\}$
is an invariant subspace of $P_{\mathcal{A}}$. 

Let us say that a function $f$ is {\em $\mathcal{A}$-symmetric} in case it is symmetric under interchanges of pairs of coordinates in $\mathcal{A}$. 
The space of  $\mathcal{A}$-symmetric functions is evidently an invariant subspace of $P_{\mathcal{A}}$,
and any (fully) symmetric function belongs to this space. 

Altogether, the supremum in (\ref{var1}) is no larger than the largest eigenvalue of $P_{\mathcal{A}}$ on the subspace of 
$\mathcal{A}$-symmetric functions that are orthogonal to each $\xi_\alpha$, $\alpha\in \mathcal{A}$.  

Therefore, let $f$ be
any eigenfunction of  $P_{\mathcal{A}}$ on this space for which the eigenvalue is non-zero. 
Since $f$ is in the range of  $P_{\mathcal{A}}$, it is necessarily of the form
$$f (v) = \sum_{\alpha\in \mathcal{A}}\varphi(\pi_\alpha(v))\ $$
for some function $\varphi$ on the ball such that $\varphi$ is orthogonal to $\eta$ on the ball.

For such an $f$, we compute
\begin{eqnarray}\label{srule3}
\langle f, P_{\mathcal A} f\rangle &=&
\frac{1}{M}\sum_{\alpha,\beta,\gamma\in {\mathcal A}} \langle \varphi \circ\pi_\beta, P_\alpha \varphi\circ\pi_\gamma\rangle_{\h}\nonumber\\
&=&
\frac{1}{M}\sum_{\alpha,\beta\in {\mathcal A}} \langle \varphi\circ\pi_\beta, 
 \varphi\circ\pi_\alpha\rangle_{\h} + 
\frac{M-1}{M}\sum_{\alpha,\beta\in {\mathcal A}}
 \langle \varphi\circ\pi_\beta,\left[ K_{N,2}\varphi\right]
 \circ\pi_\alpha\rangle_{\h}\nonumber\\
 &=&\frac{1}{M}\|f\|^2 + 
 \frac{(M-1)^2}{M}
 \left \langle \varphi + (M-1)\left[K_{N,2}\varphi\right], 
 \left[K_{N,2}\varphi\right]\right\rangle_{\K_{N,2}}  
 \end{eqnarray}
 
 $$
 \langle f, P_{\mathcal A} f\rangle \leq \frac{1}{M}\|f\|^2 +\left(\frac{8}{N(N-2)}\right) \left(\frac{M-1}{M}\right)^2 \left( 1 + (M-1)\frac{8}{N(N-2)}\right) M\|  \varphi \|^2\ .$$
 Finally, we note that, again using the spectral properties of $K_{(N,2)}$, that
 $$\|f\|^2 \geq M\|\varphi\|^2\left( 1 -  (M-1)\frac{8}{N(N-2)}\right) \ .$$
 Then, since for $0< x <1$, $(1+x) < (1-x)^{-1}$, we obtain 
 $$
 \langle f, P_{\mathcal A} f\rangle \leq 
 \left[\frac{1}{M} +\left(\frac{8}{N(N-2)}\right) \left(\frac{M-1}{M}\right)^2\left( 1 - (M-1)\frac{8}{N(N-2)}\right)^{-2}\right] \|f\|^2\ .
 $$
 Dropping the factor $((M-1)/M)^2$ and replacing $M$ by either $(N-1)/2$ or $N/2$ as appropriate, we reduce this to
 $$
 \langle f, P_{\mathcal A} f\rangle \leq 
 \left[\frac{2}{N-1} +\left(\frac{8}{N(N-2)}\right) \left( 1 - \frac{4}{N}\right)^{-2}\right] \|f\|^2\ .
 $$
 Further simplifying this, and 
  averaging over all choices
 of $\mathcal{A}$, we obtain the bound claimed in the theorem. 
 \end{proof}

We close this section by carrying out the  simple computations based on (\ref{change6}) that have been invoked in  the proof of Lemma~\ref{moments}.
First consider the case $\psi(w_1,w_2) = w_1^2$. Then
\begin{equation}\label{change8}
K_{(N,2)}\psi(w_1,w_2) = (1- |w|^2)\int_{B_2}|y_1|^2 \dd \nu_{N-2,2}(y) = 
(1- |w|^2)\int_{B_1}|y_1|^2 \dd \nu_{N-2,1}(y) = \frac{1}{N-3}(1- |w|^2)\ .
\end{equation}

Likewise, for $\psi(w_1,w_2) = w_1^4$,
\begin{multline}\label{change9}
K_{(N,2)}\psi(w_1,w_2) = (1- |w|^2)^2\int_{B_2}|y_1|^2 \dd \nu_{N-2,2}(y) =\\ 
(1- |w|^2)\int_{B_1}|y_1|^4 \dd \nu_{N-2,1}(y) = \frac{3}{(N-3)(N-1)}(1- |w|^2)^2\ .
\end{multline}

The next lemma records these computations in a form that is directly applicable  to our problem here.

\begin{lm}\label{speccomp}  Consider $v_k^2$ and $v_k^4$ as functions on $S^{N-1}(\sqrt{N})$. Then
for $k,j,\ell$ all distinct,
\begin{equation}\label{change10}
P_{\{j,\ell\}}v_k^2 =   \frac{1}{N-3}[N- (v_j^2+v_\ell^2)]
\end{equation}
and
\begin{equation}\label{change11}
P_{\{j,\ell\}}v_k^4 =  \frac{3}{(N-3)(N-1)}[N- (v_j^2+v_\ell^2)]^2\ .
\end{equation}
\end{lm}

\begin{proof}
 First define $ w$ by $ v = \sqrt{N} w$ so that $ w$ lies in the unit sphere.
 Then
 $$P_{\{j,\ell\}}v_k^2 = NP_{\{j,\ell\}}w_k^2 = N\frac{1}{N-3}(1- (w_j^2+w_\ell^2)) = 
 \frac{1}{N-3}(N- (v_j^2+v_\ell^2))\ ,$$
 where we have used (\ref{change8}). One proves  (\ref{change11}) by making analogous use of
 (\ref{change9}). 
 \end{proof}

\medskip

\section{An evaluation formula for certain infinite products}

Let $P(x)$ and $Q(x)$ be two polynomials. We are interested in calculating the infinite product
$$
\lim_{N \to \infty} \Pi_{j=M}^N \frac{P(j)}{Q(j)} \ ,
$$
and determining when the limit is non-zero. 

A necessary condition that the limit exists and is non-zero is that both polynomials are of the same degree $K$ and that
the coefficient of the two leading order terms are the same. This follows from the fact that in order for the limit
to exist and to be nonzero the factors must be of the form
$1 + R(j)$ where $R(j) = O(\frac{1}{j^2})$.
Hence we may assume that
$$
P(x) = (x-\mu_1) \cdots (x-\mu_K) \ ,  \ Q(x) = (x-\nu_1) \cdots (x-\nu_K) \ ,
$$
where $\sum \mu_n = \sum \nu_n$.
And if we seek an non-zero limit, another obvious reqirement is that none of the polynomials vanish for any $j \ge M$ since then one of the factors
is either not defined or vanishes.

\begin{thm}\label{infpro} Under the assumption stated above,
$$
\lim_{N \to \infty} \Pi_{j=M}^N \frac{P(j)}{Q(j)}   =  \Pi_{n=1}^K \frac{ \Gamma(M-\mu_n)}{ \Gamma(M-\nu_n)} \ .
$$
\end{thm}
\begin{proof}
We write
$$
\Pi_{j=M}^N (j-\mu) = \frac{\Gamma(N+1 - \mu)}{\Gamma(M-\mu)}
$$
and find
$$
\Pi_{j=M}^N \frac{P(j)}{Q(j)} = \Pi_{n=1}^K \frac{\Gamma(N+1 - \mu_n) \Gamma(M - \nu_n)}{\Gamma(M-\mu_n) \Gamma(N+1 -\nu_n)} \ .
$$
Using Stirling's formula
$$
\Gamma(x) \approx \sqrt{2\pi} x^{x-1/2} e^{-x} \ , \ x \to \infty
$$
we can write  for $N$ large
$$
\Pi_{j=M}^N \frac{P(j)}{Q(j)} \approx  \Pi_{n=1}^K \frac{(N+1-\mu_n)^{N+1/2 - \mu_n} e^{-(N+1-\mu_n)} \Gamma(M-\nu_n)}
{\Gamma(M-\mu_n) (N+1 - \nu_n)^{(N+1/2 - \nu_n)} e^{-(N+1 - \nu_n)}}
$$
which simplifies to
$$
=  \Pi_{n=1}^K \frac{(N+1)^{- \mu_n} ( 1 - \frac{\mu_n}{N+1})^{N+1/2 -\mu_n}e^{\mu_n} \Gamma(M-\mu_n)}{(N+1)^{ - \nu_n} (1 - \frac{\nu_n}{N+1})^{N+1/2 -\nu_n} e^{\nu_n} \Gamma(M-\nu_n)} \ .
$$
Using the assumption that $\sum_{n=1}^K \mu_n = \sum_{n=1}^K \nu_n$, we find for large $N$
$$
\Pi_{j=M}^N \frac{P(j)}{Q(j)} =  \Pi_{n=1}^K \frac{ ( 1 - \frac{\mu_n}{N+1})^{N+1/2 -\mu_n}e^{\mu_n} \Gamma(M-\mu_n)}{ (1 - \frac{\nu_n}{N+1})^{N+1/2 -\nu_n} e^{\nu_n} \Gamma(M-\nu_n)} \ ,
$$
which converges to
$$
\Pi_{n=1}^K \frac{ \Gamma(M-\mu_n)}{ \Gamma(M-\nu_n)}
$$

\end{proof}

  \centerline{\bf Acknowledgement}
  
The first results in this paper were obtained while the E.C. and M.L. were visiting C.M.A.F.  at the University of Lisbon. E.C. and M.L.
thank C.M.A.F. for its hospitality. All three authors have had the opportunity to continue their work together while visiting I.P.A.M. 
which the authors thank for its hospitality.

\end{document}